\documentclass[a4paper,UKenglish,cleveref, autoref, thm-restate]{lipics-v2021}
\title{Matching Cut and Variants on Bipartite Graphs of Bounded Radius and Diameter}
\titlerunning{Matching Cut and Variants on Bipartite Graphs} 

\author{Felicia Lucke}{Department of Computer Science, Durham University, Durham, United Kingdom}{felicia.c.lucke@durham.ac.uk}{https://orcid.org/0000-0002-9860-2928}{Supported by the EPSRC (Grant No.\ EP/X01357X/1)}

\authorrunning{F.\,Lucke } 
\Copyright{Felicia Lucke}

\nolinenumbers 
\hideLIPIcs

\usepackage{boxedminipage,tikz}
\usepackage{subcaption}
\usetikzlibrary{decorations.pathreplacing, calligraphy}
\usetikzlibrary{decorations.pathmorphing}
\usetikzlibrary{arrows,automata}
\usetikzlibrary{arrows.meta}
\usetikzlibrary{shapes.geometric}
\definecolor{nicered}{RGB}{204,0,0}
\definecolor{lightblue}{RGB}{153,204,255}
\definecolor{nicegreen}{RGB}{0,153,0}
 \tikzstyle{vertex}=[thin,circle,inner sep=0.cm, minimum size=1.7mm, fill=black, draw=black]
 \tikzstyle{svertex}=[thin,circle,inner sep=0.cm, minimum size=1.3mm, fill=black, draw=black]
 \tikzstyle{bvertex}=[thin,circle,inner sep=0.cm, minimum size=1.7mm, fill=lightblue, draw=lightblue]
 \tikzstyle{rvertex}=[thin,circle,inner sep=0.cm, minimum size=1.7mm, fill=nicered,draw=nicered]
 \tikzstyle{evertex}=[thin,circle,inner sep=0.cm, minimum size=1.7mm, fill=none,draw=black]
 \tikzstyle{hedge}=[thick, draw = gray]
 \tikzstyle{edge}=[thick, draw = gray]
 \tikzstyle{tedge}=[ultra thick, draw = black]
 \tikzstyle{tredge}=[ultra thick, draw = nicered]
 \tikzstyle{pedge}=[ultra thick, draw=lightblue]
 \tikzstyle{rededge}=[thick, draw = nicered]
 \tikzstyle{bluedge}=[thick, draw = lightblue]
 \tikzstyle{grnedge}=[thick, draw = nicegreen]
 \tikzstyle{edge}=[thick, draw = gray]
 \tikzstyle{br} = [decorate, ultra thick, decoration = {calligraphic brace}]
 \tikzstyle{wiggly} = [decorate, decoration = snake, thick, draw = gray,]

\newcommand\yes{\textsc{Yes}}

\newcommand{\set}[1]{\ensuremath{ \left\lbrace #1 \right\rbrace }}
\newcommand{\NP}{\textsf{NP}}
\newcommand{\p}{\textsf{P}}

\DeclareMathOperator{\dist}{{dist}}

\newcommand{\mc}{{\sc Matching Cut}}
\newcommand{\pmc}{{\sc Perfect Matching Cut}}
\newcommand{\dpm}{{\sc Disconnected Perfect Matching}}
\newcommand{\mmc}{{\sc Maximum Matching Cut}}
\newcommand{\minmc}{{\sc Minimum Matching Cut}}
\newcommand{\mdpm}{{\sc Maximum Disconnected Perfect Matching}}
\newcommand{\dc}{{\sc $d$-Cut}}

\newcommand{\naesat}{{\sc Not-All-Equal SAT}}

\newcommand{\tabref}[1]{{\footnotesize (Th.~\ref{#1})}}
\newcommand{\tabcite}[1]{{ \footnotesize (\cite{#1})}}

 \newtheorem{open}{Open Problem} 
 
\begin{document}

\maketitle

\begin{abstract}
In the \mc{} problem we ask whether a graph $G$ has a matching cut, that is, a matching which is also an edge cut of $G$. 
We consider the variants \pmc{} and \dpm{} where we ask whether there exists a matching cut equal to, respectively contained in, a perfect matching.
Further, in the problem \mmc{} we ask for a matching cut with a maximum number of edges. 
The last problem we consider is \dc{} where we ask for an edge cut where each vertex is incident to at most $d$ edges in the cut.

We investigate the computational complexity of these problems on bipartite graphs of bounded radius and diameter. Our results extend known results for \mc{} and \dpm{}. We give complexity dichotomies for \dc{} and \mmc{} and solve one of two open cases for \dpm. For \pmc{} we give the first hardness result for bipartite graphs of bounded radius and diameter and extend the known polynomial cases.

\end{abstract}

\section{Introduction}\label{s-intro}

Given a graph $G$, a \emph{matching} $M \subseteq E(G)$ is a set of edges which are pairwise non-adjacent. We say that a matching $M$ is a \emph{matching cut} of $G$ if there is a partition $A \cup B = V(G)$ of the vertices such that the edges in $M$ are exactly those edges of $G$ with one endvertex in $A$ and one in $B$.

The corresponding decision problem \mc, where we ask whether a given graph has a
matching cut, was introduced in 1970 by Graham~\cite{Gr70}.  In 1984,
Chv{\'a}tal showed that \mc{} is an \NP-complete problem. Since then,
many different variants of the problem have been introduced. 
Recall that a \emph{perfect matching} is a matching where every vertex is adjacent to an edge in the matching.
A \emph{perfect matching cut} is a perfect matching which is also a matching cut. 
In the problem \pmc{}, introduced in 1998~\cite{HT98}, we ask for the existence
of a perfect matching cut. 
We further say that a matching cut which is contained in a perfect matching is a
\emph{disconnected perfect matching}. In the corresponding decision problem \dpm,
introduced in~\cite{BP} under the name perfect matching cut, we ask for the existence of a disconnected perfect
matching.

Recently, the maximization variants of \mc{} and \dpm{} \cite{LPR24} as well as the minimization variant of \mc{} \cite{LLPR24} have been introduced. 
The \emph{size} of a matching cut is the number of edges contained in the
matching cut.
In the problems \mmc{} and \mdpm{} we ask about the maximum size of a matching
cut, respectively disconnected perfect matching of a given graph.
Similarly, in the problem \minmc{} we ask for the minimum size of a matching cut of a given graph.

The last problem we consider is a generalization of \mc{}, which is called \dc.
A graph has a \emph{$d$-cut} if there is a partition of the vertex set into sets $A$ and $B$ such that every vertex in $A$ is adjacent to at most $d$ vertices in $B$ and vice versa. The decision problem \dc{} asks whether a given graph has a $d$-cut.

Given that all these problems are \NP-hard, their computational complexity was
studied on different graph classes. Among them are for example $H$-free
graphs~\cite{BP, FLPR23, LL19, LT22, LPR23a}, chordal graphs~\cite{LL23}, graphs of bounded maximum degree \cite{BP, Ch84, GS21}, planar graphs \cite{BCD24, Bo09, BP} and graphs
of bounded radius and diameter. Further, there are exact and parameterized results. See~\cite{FLPR23} for a more detailed overview.

\begin{table}
\centering
\begin{tabular}{| l | cccc |}

\hline
\rule{0pt}{5.5mm}
& \multirow{2}*{Matching Cut}  &Disconnected&  Perfect &\multirow{2}*{$d$-Cut ($d\geq 2$)} \\
& & Perfect Matching &   Matching Cut   & \\[2mm]
\hline
\rule{0pt}{5.5mm}
\p & diam $\leq 3$ \tabcite{LL19}&  diam $\leq 3$ \tabcite{BP} &diam $\leq 3$ \tabref{t-pmc-bipdiam3}  & diam $\leq 3$ \tabref{t-dc-dc-bipdiam3} \\
\NP & diam $\geq 4$ \tabcite{LL19}& diam $\geq 4$ \tabcite{BP} & diam $\geq 8$ \tabref{t-pmc-biprad4} & diam $\geq 4$  \tabref{t-dc-dc-biprad3}  \\[2mm]
\hline
\rule{0pt}{5.5mm}
\p & radius $\leq 2$ \tabcite{LPR22} & radius $\leq 2$ (Cor.~\ref{c-p-dpm-biprad}) & 
 radius $\leq 2$ \tabcite{LPR23a}   &  radius $\leq 2$ \tabref{t-p-dc-biprad2}\\
\NP& radius $\geq 4$ \tabcite{LL19}& radius $\geq 4$ \tabcite{BP} & radius $\geq 4$ \tabref{t-pmc-biprad4}  & radius $\geq 3$ \tabref{t-dc-dc-biprad3}  \\[1.5mm]
\hline
\hline
\rule{0pt}{5.5mm}
& Minimum &         Maximum & Maximum Disconnected & \\
& Matching Cut & Matching Cut & Perfect Matching & \\[2mm]
\hline
\rule{0pt}{5.5mm}
\p & & diam $\leq 3$ \tabref{t-p-mmc-bipdiam3} & diam $\leq 3$ \tabref{t-p-mdpm-bipdiam3} & \\
\NP & diam $\geq 4$ \tabcite{LL19} & diam $\geq 4$ \tabref{t-np-mmc-biprad}  & diam $\geq 4$ \tabref{t-np-mdpm-biprad} &  \\[2mm]
\hline
\rule{0pt}{5.5mm}
\p   & & radius $\leq 2$ \tabref{t-p-mmc-biprad} & radius $\leq 2$ \tabref{t-p-mdpm-biprad} &  \\
\NP & radius $\geq 4$ \tabcite{LL19} & radius $\geq 3$ \tabref{t-np-mmc-biprad} & radius $\geq 3$ \tabref{t-np-mdpm-biprad} & \\[1.5mm]
\hline

\end{tabular}
\caption{\label{tab-intro-all-biprad}The complexity of all problems for bipartite graphs of bounded diameter (diam) and radius.}
\end{table}

The \emph{distance} of two vertices $u$ and $v$, denoted $\dist(u,v)$ is the length of a shortest path from $u$ to $v$. A graph has \emph{radius} $r$ if there is a vertex $v$ such that $\dist(v,u) \leq r$ for all $u \in V$. We call such a vertex a \emph{center vertex}. A graph has \emph{diameter} $d$ if for all vertices $u,v$ we have that $\dist(u,v) \leq d$. A graph is called \emph{bipartite} if there is a partition of the vertex set into sets $A$ and $B$ such that all edges of $G$ have one endpoint in $A$ and one in $B$.

While for all matching cut variants the computational complexity on graphs of bounded radius
and diameter has been studied, for bipartite graphs of bounded radius and
diameter, this was previously done only for \mc{} and \dpm{}.
Both problems are known to be 
\NP-complete for general graphs of radius~$3$ and diameter~$3$~\cite{BP, LL19}.
For bipartite graphs of diameter~$3$, it has been shown that \mc{}~\cite{LL19} and \dpm{}~\cite{BP} are polynomial-time solvable and \NP-complete for bipartite graphs of diameter~$4$, and thus of radius~$4$.
Note that the hardness results for \mc{} translate to \minmc{} and \mmc{} and those for \dpm{} translate to \mdpm.
It follows from the fact that \mc{} is solvable for graphs of radius~$2$~\cite{LPR22} that the same holds for bipartite graphs.
This leaves only the case of radius~$3$ open for \mc{} and the cases of radius~$2$ and $3$ for \dpm{}.

In this paper we continue the investigation by solving one of the open cases for \dpm{}.
For \dc, \mmc{} and \mdpm{} we give the complexity dichotomies for bipartite graphs of bounded radius and bounded diameter. We further give a polynomial time algorithm for \pmc{} for bipartite graphs of diameter~$3$ and we show that \pmc{} is \NP-complete for graphs of radius~$4$.
See Table~\ref{tab-intro-all-biprad} for an overview of known and new results.

In Section~\ref{s-pre} we present a transformation of the aforementioned problems into colouring problems, see also~\cite{CHLLP21, LPR23a}. This transformation will be used throughout the paper. Then, in Section~\ref{s-poly} we give our
polynomial time results which are followed by the \NP-hardness results in
Section~\ref{s-hard}.  We end our paper with some open questions.

\section{Preliminaries}\label{s-pre}
Let $G = (V,E)$ be a graph. 
The \emph{neighbourhood} of a vertex $v \in V$ is denoted by $N(v)$.
Let $Z \subseteq V$. We denote the graph \emph{induced} by $Z$ by $G[Z]$.
For integers $r,s \geq 1$, the complete graph on $r$ vertices is denoted $K_r$ and the complete bipartite graph with partition classes of size $r$ and $s$ is denoted by $K_{r,s}$. We call a complete bipartite graph a \emph{biclique}.

A \emph{red-blue colouring} of $G$ is a vertex colouring with two colours, red and blue, using every colour at least once. We say that a red-blue colouring is \emph{valid} if every vertex is adjacent to at most one vertex of the other colour, see Figure~\ref{f-example-colouring} for an example.
We call an edge with one red and one blue endvertex \emph{bichromatic} and a
vertex set where all vertices have the same colour is \emph{monochromatic}.
Further, we call a valid red-blue colouring \emph{perfect} if every vertex is adjacent to exactly one vertex of the other colour and \emph{perfect extendable} if there is a perfect matching~$M$ of $G$ such that every bichromatic edge is contained in $M$.

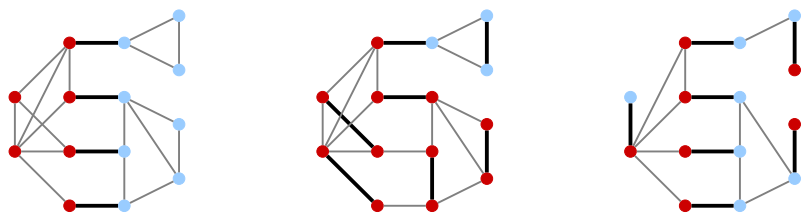
\begin{figure}
\centering
\scalebox{0.9}{
\begin{tikzpicture}

 \def\k{0.8}

	\node[rvertex] (a1) at (0,\k/2){};
	\node[rvertex] (a2) at (0,-\k/2){};
	\node[rvertex] (a3) at (\k,1.5*\k){};
	\node[rvertex] (a4) at (\k,\k/2){};
	\node[rvertex] (a5) at (\k,-3*\k/2){};
	\node[rvertex] (a5a) at (\k,-\k/2){};
	\node[bvertex] (a6) at (2*\k,1.5*\k){};
	\node[bvertex] (a7) at (2*\k,0.5*\k){};
	\node[bvertex] (a8a) at (2*\k,-0.5*\k){};
	\node[bvertex] (a8) at (2*\k,-1.5*\k){};
	\node[bvertex] (a9) at (3*\k,2*\k){};
	\node[bvertex] (a10) at (3*\k,1*\k){};
	\node[bvertex] (a11) at (3*\k,0*\k){};
	\node[bvertex] (a12) at (3*\k,-1*\k){};
	
	\draw[edge](a1)--(a2);
	\draw[edge](a1)--(a3);
	\draw[edge](a1)--(a5a);
	\draw[edge](a2)--(a3);
	\draw[edge](a2)--(a4);
	\draw[edge](a2)--(a5);
	\draw[edge](a2)--(a5a);
	\draw[edge](a3)--(a4);
	\draw[tedge](a3)--(a6);
	\draw[tedge](a4)--(a7);
	\draw[tedge](a5)--(a8);
	\draw[tedge](a5a)--(a8a);
	\draw[edge](a6)--(a9);
	\draw[edge](a6)--(a10);
	\draw[edge](a7)--(a8a);
	\draw[edge](a7)--(a11);
	\draw[edge](a7)--(a12);
	\draw[edge](a8)--(a12);
	\draw[edge](a8a)--(a8);
	\draw[edge](a9)--(a10);
	\draw[edge](a11)--(a12);

	\begin{scope}[shift= {(4.5,0)}]
	\node[rvertex] (a1) at (0,\k/2){};
	\node[rvertex] (a2) at (0,-\k/2){};
	\node[rvertex] (a3) at (\k,1.5*\k){};
	\node[rvertex] (a4) at (\k,\k/2){};
	\node[rvertex] (a5) at (\k,-3*\k/2){};
	\node[rvertex] (a5a) at (\k,-\k/2){};
	\node[bvertex] (a6) at (2*\k,1.5*\k){};
	\node[rvertex] (a7) at (2*\k,0.5*\k){};
	\node[rvertex] (a8a) at (2*\k,-0.5*\k){};
	\node[rvertex] (a8) at (2*\k,-1.5*\k){};
	\node[bvertex] (a9) at (3*\k,2*\k){};
	\node[bvertex] (a10) at (3*\k,1*\k){};
	\node[rvertex] (a11) at (3*\k,0*\k){};
	\node[rvertex] (a12) at (3*\k,-1*\k){};
	
	\draw[edge](a1)--(a2);
	\draw[edge](a1)--(a3);
	\draw[tedge](a1)--(a5a);
	\draw[edge](a2)--(a3);
	\draw[edge](a2)--(a4);
	\draw[tedge](a2)--(a5);
	\draw[edge](a2)--(a5a);
	\draw[edge](a3)--(a4);
	\draw[tedge](a3)--(a6);
	\draw[tedge](a4)--(a7);
	\draw[edge](a5)--(a8);
	\draw[edge](a5a)--(a8a);
	\draw[edge](a6)--(a9);
	\draw[edge](a6)--(a10);
	\draw[edge](a7)--(a8a);
	\draw[edge](a7)--(a11);
	\draw[edge](a7)--(a12);
	\draw[edge](a8)--(a12);
	\draw[tedge](a8a)--(a8);
	\draw[tedge](a9)--(a10);
	\draw[tedge](a11)--(a12);
	
	\end{scope}

\begin{scope}[shift = {(9,0)}]
	\node[bvertex] (a1) at (0,\k/2){};
	\node[rvertex] (a2) at (0,-\k/2){};
	\node[rvertex] (a3) at (\k,1.5*\k){};
	\node[rvertex] (a4) at (\k,\k/2){};
	\node[rvertex] (a5) at (\k,-3*\k/2){};
	\node[rvertex] (a5a) at (\k,-\k/2){};
	\node[bvertex] (a6) at (2*\k,1.5*\k){};
	\node[bvertex] (a7) at (2*\k,0.5*\k){};
	\node[bvertex] (a8a) at (2*\k,-0.5*\k){};
	\node[bvertex] (a8) at (2*\k,-1.5*\k){};
	\node[bvertex] (a9) at (3*\k,2*\k){};
	\node[rvertex] (a10) at (3*\k,1*\k){};
	\node[rvertex] (a11) at (3*\k,0*\k){};
	\node[bvertex] (a12) at (3*\k,-1*\k){};
	
	\draw[tedge](a1)--(a2);
	\draw[edge](a2)--(a3);
	\draw[edge](a2)--(a4);
	\draw[edge](a2)--(a5);
	\draw[edge](a2)--(a5a);
	\draw[edge](a3)--(a4);
	\draw[tedge](a3)--(a6);
	\draw[tedge](a4)--(a7);
	\draw[tedge](a5)--(a8);
	\draw[tedge](a5a)--(a8a);
	\draw[edge](a6)--(a9);
	\draw[edge](a7)--(a8a);
	\draw[edge](a7)--(a12);
	\draw[edge](a8)--(a12);
	\draw[edge](a8a)--(a8);
	\draw[tedge](a9)--(a10);
	\draw[tedge](a11)--(a12);
	\end{scope}

\end{tikzpicture}}
\caption{\label{f-example-colouring} A valid red-blue colouring of value~$4$ (left), a perfect extendable red-blue colouring of value~$1$ (middle) and a perfect red-blue colouring (right).}
\end{figure}

The \emph{value} $\nu$ of a red-blue colouring is the number of bichromatic edges. A red-blue colouring with maximum (respectively minimum) value is a \emph{maximum valid red-blue colouring} (respectively \emph{minimum valid red-blue colouring}).
A red-blue colouring is a \emph{red-blue $d$-colouring} if every vertex is adjacent to at most $d$ vertices of the other colour.
We make the following straightforward observations, which has been used in previous works, see e.g.~\cite{LPR23a}.

\begin{observation}\label{o-cut-colouring}
Let $G$ be a graph. Then the following holds.
\begin{itemize}
\item $G$ has a matching cut of size $\nu$ if and only if $G$ has a valid red-blue colouring of value $\nu$.
\item $G$ has a disconnected perfect matching of size $\nu$ if and only if $G$ has a perfect extendable red-blue colouring of value~$\nu$.
\item $G$ has a perfect matching cut if and only if $G$ has a perfect red-blue colouring.
\item $G$ has a $d$-cut, for $d \geq 1$, if and only if $G$ has a red-blue $d$-colouring.
\end{itemize}
\end{observation}

\begin{observation}\label{o-c-all-kr-mono}
For every $d\geq 1$, every complete graph $K_{r}$ with $r \geq 2d+1$ and every complete bipartite graph $K_{r,s}$ with $\min\{r,s\} \geq 2d$ and $\max\{r,s\} \geq 2d+1$ is monochromatic in every red-blue $d$-colouring.
\end{observation}

\noindent
Let $S, T \subseteq V$ be two disjoint vertex sets. We say that a \emph{red-blue $(S,T)$-$d$-colouring} is a red-blue $d$-colouring where the vertices of $S$ are coloured red and those of $T$ are coloured blue. When considering the case of $d=1$, we also use the name of \emph{red-blue $(S,T)$-colouring}.
A \emph{perfect red-blue $(S,T)$-colouring} is a red-blue $(S,T)$-colouring of $G$ where every vertex in $S$ and $T$ for which all neighbours are coloured has exactly one neighbour of the other colour.
We call $(S,T)$ a \emph{precoloured pair} and we define the following rules:
\begin{description}
\item [\bf R1.] Return {\tt no} (i.e., $G$ has no red-blue $(S,T)$-colouring) if 
	\begin{itemize}
		\item[i)] a vertex $v\in Z$ is adjacent to $d+1$ vertices in $S$ and to $d+1$ vertices in $T$.
		\item[ii)] a vertex $v \in S$ has $d+1$ neighbours in $T$.
		\item[iii)] a vertex $v \in T$ has $d+1$ neighbours in $S$.
	\end{itemize}

\item [{\bf R2.}] Let $v\in Z$.
\begin{itemize}
\item [(i)] If $v$ is adjacent to $d+1$ vertices of $S$, then move $v$ from $Z$ to $S$. 
\item [(ii)] If $v$ is adjacent to $d+1$ vertices of $T$, then move $v$ from $Z$ to $T$. 
\end{itemize}
\end{description}

\noindent
For perfect red-blue colourings we add two more rules.
\begin{description}
\item [{\bf R3.}] Return {\tt no} (i.e., $G$ has no perfect red-blue $(S,T)$-colouring) if
\begin{itemize}
\item [(i)] for a vertex $v \in S$,  $N(v) \subseteq S$, i.e.~$v$ cannot have a neighbour in $T$. 
\item [(ii)] for a vertex $v \in T$,  $N(v) \subseteq T$, i.e.~$v$ cannot have a neighbour in $S$. 
\end{itemize}

\item [{\bf R4.}] Let $v \in Z$.
\begin{itemize}
\item [(i)] If $v$ has a neighbour in $S$ with a neighbour in $T$, add $v$ to $S$.
\item [(ii)] If $v$ has a neighbour in $T$ with a neighbour in $S$, add $v$ to $T$.
\end{itemize}
\end{description}

%

\noindent
Given a precoloured pair we will apply these rules exhaustively.  We then say that $G$ is \emph{colour-processed}. The next lemmas show that we may always assume without loss of generality that $G$ is colour-processed.
We omit the proof of the first, which was already given in~\cite{LMPS24}.

\begin{lemma}[\cite{LMPS24}]\label{l-all-precol}
Let $G$ be a connected graph with a precoloured pair $(S,T)$.  It is possible, in polynomial time, to either colour-process $(S,T)$ using rules R1 and R2 or to find that $G$ has no red-blue $(S,T)$-$d$-colouring.
\end{lemma}

\noindent
 For the proof of the second lemma, note that rules R1 -- R4 are a subset of the rules given in~\cite{LPR23a}. Thus, the following follows immediately from~\cite{LPR23a}. 
\begin{lemma}\label{l-pmc-precol}
Let $G$ be a connected graph with a precoloured pair $(S,T)$.  It is possible, in polynomial time, to either colour-process $(S,T)$ using rules R1 -- R4 or to find that $G$ has no perfect red-blue $(S,T)$-colouring.
\end{lemma}

%

\noindent
For perfect red-blue colourings we will also make use of the following result.
\begin{lemma}[\cite{LPR23a}]\label{l-pmc-monocol}
Let $G$ be a colour-processed graph and $Z$ be the set of uncoloured vertices. We can find in polynomial time a perfect red-blue colouring of $G$ where every connected component of $G[Z]$ is monochromatic or determine that no such colouring exists.
\end{lemma}

\noindent
To obtain our results for \mmc{} and \mdpm{} we need two more lemmas from~\cite{LPR24}.

 \begin{lemma}[\cite{LPR24}]\label{l-p-mmc-applplesnik}
Let $G=(V,E)$ be a connected graph with a precoloured pair $(S,T)$. If $V\setminus (S\cup T)$ is an independent set, then it is possible to find in polynomial time either a maximum valid red-blue $(S,T)$-colouring of~$G$, or conclude that $G$ has no valid red-blue $(S,T)$-colouring.
\end{lemma}

 \begin{lemma}[\cite{LPR24}]\label{l-p-mdpm-applplesnik}
Let $G=(V,E)$ be a connected graph with a precoloured pair $(S,T)$. If $V\setminus (S\cup T)$ is an independent set, then it is possible to find in polynomial time either a maximum perfect-extendable red-blue $(S,T)$-colouring of $G$, or conclude that $G$ has no perfect-extendable red-blue $(S,T)$-colouring.
\end{lemma}

\section{Polynomial-Time Results}\label{s-poly}

\noindent
In~\cite{LL19}, Le and Le showed that \mc{} is solvable in polynomial time for bipartite graphs of diameter~$3$. We adapt their proof to show that \pmc{} is solvable in polynomial time for bipartite graphs of diameter~$3$. 

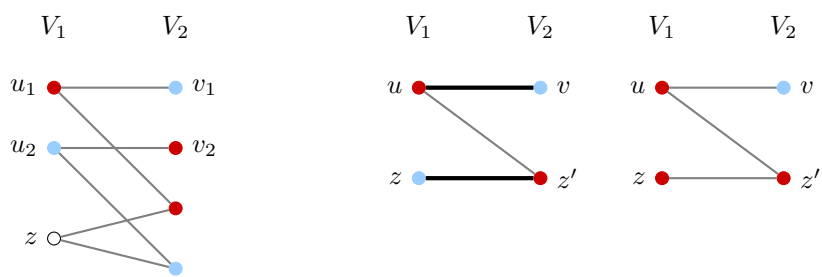
\begin{figure}
\centering
\begin{tikzpicture}
\begin{scope}[scale = 0.8]
\begin{scope}
\node[rvertex, label = left:$u_1$] (u1) at (0,3){};
\node[bvertex, label = left:$u_2$] (u2) at (0,2){};
\node[evertex, label = left: $z$] (z) at (0,0.5){};
\node[bvertex, label = right: $v_1$] (v1) at (2,3){};
\node[rvertex, label = right: $v_2$] (v2) at (2,2){};
\node[rvertex] (z1) at (2,1){};
\node[bvertex] (z2) at (2,0){};

\draw[edge] (u1) -- (v1);
\draw[edge] (u1) -- (z1);
\draw[edge] (u2) -- (v2);
\draw[edge] (u2) -- (z2);
\draw[edge] (z) -- (z1);
\draw[edge] (z) -- (z2);

\node[] (V1) at (0,4){$V_1$};
\node[] (V2) at (2,4){$V_2$};
\end{scope}

\begin{scope}[shift = {(6,0)}]
\node[rvertex, label = left:$u$] (u1) at (0,3){};
\node[bvertex, label = left: $z$] (z) at (0,1.5){};
\node[bvertex, label = right: $v$] (v1) at (2,3){};
\node[rvertex, label = right: $z'$] (z1) at (2,1.5){};

\draw[tedge] (u1) -- (v1);
\draw[edge] (u1) -- (z1);
\draw[tedge] (z) -- (z1);

\node[] (V1) at (0,4){$V_1$};
\node[] (V2) at (2,4){$V_2$};
\end{scope}

\begin{scope}[shift = {(10,0)}]
\node[rvertex, label = left:$u$] (u1) at (0,3){};
\node[rvertex, label = left: $z$] (z) at (0,1.5){};
\node[bvertex, label = right: $v$] (v1) at (2,3){};
\node[rvertex, label = right: $z'$] (z1) at (2,1.5){};

\draw[edge] (u1) -- (v1);
\draw[edge] (u1) -- (z1);
\draw[edge] (z) -- (z1);

\node[] (V1) at (0,4){$V_1$};
\node[] (V2) at (2,4){$V_2$};
\end{scope}
\end{scope}
\end{tikzpicture}
\caption{\label{f-p-pmc-bipdiam3} An illustration of the colouring in Phase 1 (left) and the colouring in Phase 2 (middle, right) in the proof of Theorem~\ref{t-pmc-bipdiam3}.  Note that for Phase~2 the colouring in the middle leads to a contradiction to the fact that Phase~1 failed.}
\end{figure}

\begin{observation}[\cite{LL19}]\label{o-pr-all-bipdiam}
Let $G = (V,E)$ be a bipartite graph of diameter at most~$3$ with partition classes $V_1$ and $V_2$. 
Let $u, v \in V_i$, for some $i \in \{1,2\}$. Then, $u$ and $v$ have a common neighbour.
\end{observation}

\begin{theorem}\label{t-pmc-bipdiam3}
\pmc{} is polynomial-time solvable for bipartite graphs of diameter at most~$3$.
\end{theorem}
\begin{proof}
Let $G = (V,E)$ be a connected bipartite graph of diameter at most~$3$ with partition classes~$V_1$ and $V_2$.
Note that in a bipartite graph every perfect matching cut has size $|V|/2$.
We may assume that $|V_1|, |V_2| \geq 2$, since otherwise it is trivial to decide if $G$ has a perfect matching cut or not.
This implies that any perfect matching cut in $G$ has at least $2$ edges.
We apply Observation~\ref{o-cut-colouring} and check whether the graph has a perfect red-blue colouring.
The proof consists of two phases. In Phase~1 we only search for perfect matching cuts with a special structure. In Phase~2 we consider the case where no perfect matching cut was found in Phase~1.

\medskip\noindent\textbf{Phase 1.} We branch over all $O(n^4)$ options of choosing two edges $u_1v_1, u_2v_2 \in E$, with $u_1,u_2 \in V_1$ and $v_1,v_2 \in V_2$, and we colour $u_1$ and $v_2$ red and $v_1$ and $u_2$ blue.
We define a precoloured pair $(S,T)$ with $S = \{u_1, v_2\}$ and $T = \{v_1, u_2\}$.
We apply rules R1--R4 exhaustively. If we find a no-answer, then we discard the option by Lemma~\ref{l-pmc-precol}. Else we may assume that $G$ is colour-processed. Let $Z$ be the set of uncoloured vertices of $G$.

\begin{claim}
Every connected component of $G[Z]$ in every perfect red-blue $(S,T)$-colouring of $G$ is monochromatic.
\end{claim}
\begin{claimproof}
Let $z \in Z$. Assume first that $z \in V_1$. By Observation~\ref{o-pr-all-bipdiam} we know that $z$ and~$u_1$ have a common neighbour. Further, since $u_1 \in S^*$, we get from Rule R4 that the common neighbour of $z$ and $u_1$ is red. Similarly, $z$ and $u_2$ have a common neighbour which is blue. So every uncoloured vertex in $V_1$ has both a red and a blue neighbour, see Figure~\ref{f-p-pmc-bipdiam3} (left).
Note that by symmetry the same follows if $z \in V_2$.
Thus, every vertex $z \in Z$ has at least one red and one blue neighbour. 
Let $zz'$ be an edge of $G[Z]$. If we colour $z$ blue and $z'$ red, then $z$ has two red neighbours but is blue, a contradiction.
We conclude that the claim is true.
\end{claimproof}

\noindent
We now apply Lemma~\ref{l-pmc-monocol}. This takes polynomial time. If this leads to a no-answer, then we may discard the branch. Otherwise we return the found perfect red-blue colouring of $G$.

\medskip
\noindent
If one of the considered branches lead to a perfect red-blue colouring, we are done. Otherwise we may assume that there does not exist a perfect red-blue colouring of $G$ with two edges $u_1v_1, u_2v_2 \in E$, where $u_1,u_2 \in V_1$ and $v_1,v_2 \in V_2$, such that $u_1$ and $v_2$ are coloured red and $v_1$ and $u_2$ are coloured blue.
That is, Phase~1 failed and we start Phase~2.

\medskip
\noindent\textbf{Phase 2.} We branch over all options of choosing an edge $uv \in E$, with $u \in V_1$ and $v \in V_2$. We colour $u$ red and $v$ blue.
We define a precoloured pair $(S,T)$ with $S = \{u\}$ and $T = \{v\}$.
We apply rules R1--R4 exhaustively. If we find a no-answer, then we discard the option by Lemma~\ref{l-pmc-precol}. Else we may assume that $G$ is colour-processed.

Consider an uncoloured vertex $z \in V_1$. By Observation~\ref{o-pr-all-bipdiam}, we know that $z$ has a common neighbour $z'$ with $u$. Note that $z'$ was coloured red by propagation. 
If we colour $z$ blue,  then the existence of a perfect red-blue colouring of $G$ where we have coloured the two edges $zz'$ and $uv$ as described above would contradict the assumption that Phase~1 failed, see Figure~\ref{f-p-pmc-bipdiam3}~(middle).
Thus, $z$ has to be red, see Figure~\ref{f-p-pmc-bipdiam3}~(right).
Similarly, if we consider an uncoloured vertex $z \in V_2$ we conclude that $z$ has to be blue.
Hence, we colour all remaining vertices of $G$ and check in polynomial time if this colouring is a perfect red-blue colouring. 
If this is the case, we are done, otherwise we discard the branch and consider the next.

\medskip
\noindent
If, after considering all branches in Phase~2, we did not find a perfect red-blue colouring, then $G$ does not have a perfect red-blue colouring and thus no perfect matching cut.
The correctness of our algorithm follows from its description. As the total number of branches is $O(n^4)$ and we can process each branch in polynomial time, the total running time of our algorithm is polynomial. Hence, we have proven the theorem.
\end{proof}

\noindent
For \dc{} we can show the same result, which is however best possible, see Theorem~\ref{t-dc-dc-biprad3}.
In the case of $d=1$, this was already known from~\cite{LL19}, our proof is an alternative proof.

\begin{theorem}\label{t-dc-dc-bipdiam3}
 For $d\geq 1$, \dc{} is polynomial-time solvable for bipartite graphs of diameter at most~$3$.
 \end{theorem}

\begin{proof}
Let $G=(V,E)$ be a bipartite graph of diameter at most~$3$. Note that this implies that $G$ is connected.
We use Observation~\ref{o-cut-colouring} and try to construct a red-blue $d$-colouring of~$G$.
Let $v\in V$. Without loss of generality, we may colour $v$ blue.
Since $v$ has at most $d$ red neighbours in any red-blue $d$-colouring, we can branch over all $O(n^{d})$ options to colour the neighbourhood of $v$.
We consider each option separately.

Let $R, B \subseteq V$ be the red and the blue set, respectively.
We define a precoloured pair $(S',T')$, where $S' = \{u\; |\; u \in N(v), u \in R\}$ and  $T' = \{v\} \cup  \{u\; |\; u \in N(v), u \in B\}$. We use Lemma~\ref{l-all-precol} to colour-process $(S',T')$ in polynomial time, resulting in a pair $(S,T)$.
Let $Z = V\setminus (S \cup T)$ be the set of uncoloured vertices.
As $(S,T)$ is colour-processed, every vertex in $Z$ has at most $d$ red and $d$ blue neighbours and thus at most $2d$ coloured neighbours. 
Note further that $G[Z]$ and thus, every connected component of $G[Z]$ is bipartite.
Let $A$ be the set of uncoloured vertices at distance $2$ from $v$ and let $B$ be the set of uncoloured vertices at distance $3$ from $v$. Hence, $A$ and $B$ are the partition classes of $G[Z]$. 
We distinguish between the following cases:

\medskip
\noindent
{\bf Case 1.} $G[Z]$ consists of at least two connected components.\\
Let $Z_1, \dots, Z_r$ be the connected components of $G[Z]$.
Consider first the case of two components each of size $1$. In this case we branch over all colourings of $Z$.
Otherwise we may assume without loss of generality that $|A| \geq 2$ and that $A$ contains vertices in at least two connected components of $G[Z]$, say $Z_1$ and $Z_2$. 

Let $v_1 \in Z_1^A$ and $v_2 \in Z_2^A$, where $Z_i^A = Z_i \cap A$, for $i =1,2$.
By Observation~\ref{o-pr-all-bipdiam}, $v_1$ has a common neighbour with every vertex in $Z_2^A, \dots, Z_r^A$. 
Let $N_1$ be the set of common neighbours of $v_1$ and $Z_2^A, \dots, Z_r^A$.
Then, $N_1 \subseteq S \cup T$ and thus, $|N_1| \leq 2d$, since $v_1$ has at most $2d$ coloured neighbours.
Similarly, let $N_2$ be the set of common neighbours of $v_2$ and $Z_1^A$. Then, $|N_2| \leq 2d$ as well.
Hence, we branch over all $O(n^{4d^2})$ options to colour $N_{G[Z]}(N_1 \cup N_2) = A$.

Therefore only the vertices in $B$ remain uncoloured.
If $|B| = 1$ we branch over the two options of colouring $B$. Otherwise, since $B$ is an independent set and everything except $B$ is coloured, we can apply the arguments we used before to colour $A$.
This leads to another $O(n^{4d^2})$ branches and the set $B$ will be coloured.

\medskip
\noindent
{\bf Case 2.} $G[Z]$ is connected.\\ 
We first assume that there does not exist a vertex $u \in A$ such that $\dist_{G[Z]}(u, z) \leq 2$ for all $z \in A$.
Let $u \in A$ and let $U^A = \{z \in A : \dist_{G[Z]}(u,z) > 2\}$. Note that, by assumption, $U^A \neq \emptyset$.
By Observation~\ref{o-pr-all-bipdiam} we find that $u$ has a common neighbour with every vertex in $U^A$. These common neighbours are not contained in $Z$ and thus they are in $S \cup T$, the set of coloured vertices.
Since $u$ is uncoloured, it has at most $2d$ coloured neighbours. Each of them has at most $d$ neighbours of the other colour.  So we can branch over all $O(n^{2d^2})$ options to colour $U^A$.
If any colouring of $U^A$ leads to a vertex having more than $d$ neighbours of the other colour, we discard the branch.

If there is no vertex $v \in B$ such that $\dist_{G[Z]}(v, z) \leq 2$ for all $z \in B$, then we pick any vertex $v \in B$, define the set $U^B = \{z \in B : \dist_{G[Z]}(u,z) > 2\}$ and proceed as in the case for $A$.
Thus, we may now assume that both $U^A$ and $U^B$ are coloured and that the graph $G[Z \setminus (U^A \cup U^B)]$ remains uncoloured.  
Let $Z = Z \setminus  (U^A \cup U^B)$, that is we update $Z$ such that it contains only uncoloured vertices.

So we may assume (possibly after some branching) that there exist vertices $u \in A$ and $v \in B$ such that $\dist_{G[Z]}(u, z) \leq 2$ for all $z \in A$ and $\dist_{G[Z]}(v, z) \leq 2$ for all $z \in B$.
We branch over all $O(n^{2d})$ colourings of $\{u\} \cup N_{G[Z]}(u)$ and $\{v\} \cup N_{G[Z]}(v)$.
Due to our assumption on $u$ and $v$, we get that $\{u,v\} \cup N_{G[Z]}(u)\cup N_{G[Z]}(v)$ is a dominating set of $G[Z]$.
Hence, every uncoloured vertex got a new coloured neighbour. 
We colour-process the pair of red and blue sets, which takes polynomial time by Lemma~\ref{l-all-precol}. 
This results in a new uncoloured set of vertices $Z'$.
If there is a coloured vertex with $d+1$ neighbours of the other colour or an uncoloured vertex with $d+1$ neighbours of each colour we discard the branch. Otherwise we proceed by choosing either Case~1 or Case~2.

\medskip
\noindent
If at some point we apply Case~1, we are done. Every time we apply Case~2, all vertices that remain uncoloured get a new coloured neighbour. 
Hence, we can apply Case~2 at most $2d$ times, since afterwards every uncoloured vertex has more than $2d$ coloured neighbours and thus, they will be coloured by colour-processing.

The correctness of our algorithm follows from its description. We now discuss its running time.
We always have $O(n^d)$ branches in the beginning. Case 1 leads to at most $O(n^{8d^2})$ branches and Case 2 to $O(n^{4d^2 + 2d})$ branches. After one application of Case 1 the whole graph is coloured.
Hence, in the worst case we apply Case~2 $2d$ times before applying Case~1 once. Thus, we obtain
$$O\left(n^d \cdot \left(n^{8d^2}\right)\cdot  \left(n^{4d^2+2d}\right)^{^{2d}}  \right) = O\left(n^{8d^3+ 12d^2 + d}\right)$$
 branches to consider in the worst case. Since by Lemma~\ref{l-all-precol} the colour-processing can be done in polynomial time, the algorithm therefore runs in polynomial time.
\end{proof}

\noindent
We can adjust this proof to hold for \mmc{} by the following small change.  We consider the proof for $d=1$ and remember the value of every valid red-blue colouring that we get. In the end, we pick the valid red-blue colouring with the largest value to obtain a maximum valid red-blue colouring of $G$. This leads to the following theorem.

\begin{theorem}\label{t-p-mmc-bipdiam3}
\mmc{} is solvable in polynomial time for bipartite graphs of diameter at most~$3$.
\end{theorem}

\noindent
Similarly, we can modify the proof of Theorem~\ref{t-dc-dc-bipdiam3} to also hold for \mdpm.
We only make the following change.
Consider the proof for $d=1$.
Recall that every maximum perfect extendable red-blue colouring is also a valid red-blue colouring. 
Check for every valid red-blue colouring obtained by the algorithm if it is also a perfect extendable red-blue colouring. 
If this is the case we remember its value.
In the end, we pick the perfect extendable red-blue colouring with the largest value to obtain a maximum perfect extendable red-blue colouring. This leads to the following theorem.

\begin{theorem}\label{t-p-mdpm-bipdiam3}
\mdpm{} is solvable in polynomial time for bipartite graphs of diameter at most~$3$.
\end{theorem}

\noindent
It has been shown in~\cite{LPR22} that \mc{} is solvable in polynomial time for graphs of radius at most~$2$. \dc{}, for $d \geq 2$, in contrary is \NP-complete in this case, see~\cite{LMPS24}. However, for bipartite graphs we can show that \dc{} is solvable in polynomial time when the radius is at most~$2$. This result is best possible, see Theorem~\ref{t-dc-dc-biprad3}. The following proof works also for $d=1$ and is an alternative proof in that case.

\begin{theorem}\label{t-p-dc-biprad2}
For $d \geq 1$, \dc{} is polynomial-time solvable for bipartite graphs of radius at most~$2$.
\end{theorem}

\begin{proof}
Let $G = (V,E)$ be a bipartite graph of radius at most~$2$ on $n$ vertices. Note that this implies that $G$ is connected.
By Observation~\ref{o-cut-colouring}, it suffices to find a red-blue $d$-colouring of $G$.
Let $v \in V$ be a center vertex, that is a vertex at distance at most~$2$ to every other vertex in $V$. We can find $v$ in polynomial time.
Without loss of generality, we may colour $v$ blue.
Since $v$ has at most $d$ red neighbours in any red-blue $d$-colouring, we branch over all $O(n^d)$ options to colour the neighbourhood of $v$. 
In the case where we colour all neighbours of $v$ blue, we branch over all $O(n)$ options to colour one of the remaining vertices red.
We further branch over all $O(n^{d^2})$ options of colouring the neighbourhood of the at most $d$ red vertices.

Since $G$ has radius~$2$, every uncoloured vertex is adjacent to a vertex in $N(v)$. Further, since the neighbourhood of all red vertices in $N(v)$ is coloured, every uncoloured vertex has neighbours only in $N(v)$ of which at least one is blue and, since $G$ is bipartite, none is red.
Thus, we can safely colour all uncoloured vertices blue.
If this leads to a contradiction, we discard the branch and consider the next, otherwise we found a red-blue $d$-colouring and thus a $d$-cut of~$G$.
The correctness of the algorithm follows from its description.
Every branch can be processed in polynomial time and we consider $O(n^{d^2 + d + 1})$ branches.
Hence, our algorithm runs in polynomial time.
\end{proof}

\begin{theorem}\label{t-p-mmc-biprad}
\mmc \ is solvable in polynomial time for bipartite graphs of radius at most~$2$.
\end{theorem}
\begin{proof}
Let $G = (V,E)$ be a bipartite graph of radius at most~$2$ on $n$ vertices.
If $G$ has radius~$1$, then $G$ is a star and the problem is trivial to solve.
Assume that $G$ has radius~$2$.
By Observation~\ref{o-cut-colouring}, it suffices to find a maximum valid red-blue colouring of $G$.
Since $G$ has radius~$2$, there is a dominating star $D$ with center vertex $v$ and leafs $\{u_1, \dots, u_k\}$, for some integer $k \geq 1$.
Moreover, since $G$ is bipartite, this star is induced.
Without loss of generality, we may colour $v$ blue. At most one neighbour of $v$ can be blue in any maximum valid red-blue colouring of $G$, and thus, we branch over all $O(n)$ possible colourings of $N(v) \subseteq V(D)$.

Consider first the case where $D$ is monochromatic, say all vertices in $D$ are coloured blue.
By definition, a maximum valid red-blue colouring of $G$ has to contain at least one red vertex.
So we branch over all $O(n)$ options of choosing a vertex $w$ in $V-D$ to be coloured red.
Recall that $D$ is dominating so $w$ has at least one neighbour in $D$.

In both cases, regardless of wether $D$ is monochromatic, there is one red vertex and all other coloured vertices are coloured blue. Let $T$ be the set of blue vertices and $S$ the set of red vertices, then $(S,T)$ is a precoloured pair.

By Lemma~\ref{l-all-precol} we either find in polynomial time that $G$ has no valid red-blue $(S,T)$-colouring, and we discard the branch, or we colour-processed $G$.
Suppose the latter case holds.
Since $G$ is bipartite, we have that $N(V(D))$ is an independent set and thus, $Z = V\setminus(S \cup T)$ is independent. 
This allows us to apply Lemma~\ref{l-p-mmc-applplesnik} and we either find in polynomial time a red-blue colouring of $G$ that is a maximum red-blue $(S,T)$-colouring or that no such colouring exists. In the latter case, we discard the branch.

If somewhere in the above process we discarded a branch, we consider the next one. If we did not discard the branch, then we remember the value of the maximum red-blue $(S,T)$-colouring that we found. Afterwards, we pick one with the largest value to obtain a maximum valid red-blue colouring of $G$.

The correctness of our branching algorithm follows from its description. The running time is polynomial: each branch takes polynomial time to process, and the number of branches is $O(n)$. This completes our proof.
\end{proof}

\noindent
 The following result is proven in exactly the same way as Theorem~\ref{t-p-mmc-biprad} after replacing 
  Lemma~\ref{l-p-mmc-applplesnik} by Lemma~\ref{l-p-mdpm-applplesnik}.

\begin{theorem}\label{t-p-mdpm-biprad}
\mdpm \ is solvable in polynomial time for bipartite graphs of radius at most~$2$.
\end{theorem}

\noindent
This implies that the same result holds for \dpm{} and therefore solves one of the two open cases for \dpm{} on bipartite graphs of bounded radius and diameter.

\begin{corollary}\label{c-p-dpm-biprad}
\dpm \ is solvable in polynomial time for bipartite graphs of radius at most~$2$.
\end{corollary}

\section{Hardness Results}\label{s-hard}
In this section we give our \NP-hardness results for \pmc, \dc{} and \mmc.
The first two reductions are from \naesat.
In this problem we are given clauses $C_1, \dots, C_m$ over variables $x_1, \dots, x_n$.
We ask whether a truth assignment exists, such that in every clause there is at least one true and one false literal.
This is well known to be \NP-complete, even under the following restrictions.

\begin{theorem}[\cite{Sc78}]\label{t-np-none-naesat}
\naesat{} is \NP-complete even for instances, in which each literal occurs only positively and
\begin{itemize}
\item in two or three different clauses, and in which each clause contains either two or three literals.
\item in which each clause contains three literals.
\end{itemize}
\end{theorem}

\noindent
For the following proof we modify a construction from~\cite{LT22}.
The graph in the original construction is bipartite and we modify it to have radius at most~$4$. Note that this implies the diameter to be bounded by~$8$. A better bound for the diameter cannot be achieved with our construction.

\begin{theorem}\label{t-pmc-biprad4}
\pmc{} is \NP-complete for bipartite graphs of radius~$4$.
\end{theorem}
\begin{proof}
We reduce from \naesat. Let $(X,\mathcal{C})$ be an instance of this problem, where $X = \{x_1,x_2,\dots,x_n\}$ and ${\cal C} = \{C_1, C_2, \dots , C_m\}$. By Theorem~\ref{t-np-none-naesat}, we know that \naesat{} is \NP-complete even for instances where each literal occurs only positively and every clause contains three literals.

\begin{figure}
\centering
\scalebox{0.9}{
\begin{tikzpicture}
\begin{scope}[scale = 0.9]
\begin{scope}[scale = 0.8]
\def\k{1}

\node[vertex](v1) at (0,0){};
\node[vertex, label=left: $k_j^2$](v2) at (0,2){};
\node[vertex, label=below: $k_j^1$](v3) at (2,0){};
\node[vertex](v4) at (2,2){};
\node[vertex, label = left:$a_j$](v5) at (\k,\k){};
\node[vertex](v6) at (\k, 2+\k){};
\node[vertex](v7) at (2+\k, \k){};
\node[vertex, label=above: $k_j^3$](v8) at (2+\k, 2+\k){};

\draw[edge](v1) -- (v2);
\draw[edge](v1) -- (v3);
\draw[edge](v1) -- (v5);
\draw[edge](v2) -- (v4);
\draw[edge](v2) -- (v6);
\draw[edge](v3) -- (v4);
\draw[edge](v3) -- (v7);
\draw[edge](v4) -- (v8);
\draw[edge](v5) -- (v6);
\draw[edge](v6) -- (v8);
\draw[edge](v7) -- (v5);
\draw[edge](v8) -- (v7);
\end{scope}

\begin{scope}[scale = 0.8, shift = {(5.5,0)}]
\def\k{1}

\node[bvertex](v1) at (0,0){};
\node[bvertex, label=left: $k_j^2$](v2) at (0,2){};
\node[bvertex, label=below: $k_j^1$](v3) at (2,0){};
\node[bvertex](v4) at (2,2){};
\node[rvertex, label = left:$a_j$](v5) at (\k,\k){};
\node[rvertex](v6) at (\k, 2+\k){};
\node[rvertex](v7) at (2+\k, \k){};
\node[rvertex, label=above: $k_j^3$](v8) at (2+\k, 2+\k){};

\draw[edge](v1) -- (v2);
\draw[edge](v1) -- (v3);
\draw[tedge](v1) -- (v5);
\draw[edge](v2) -- (v4);
\draw[tedge](v2) -- (v6);
\draw[edge](v3) -- (v4);
\draw[tedge](v3) -- (v7);
\draw[tedge](v4) -- (v8);
\draw[edge](v5) -- (v6);
\draw[edge](v6) -- (v8);
\draw[edge](v7) -- (v5);
\draw[edge](v8) -- (v7);
\end{scope}

\begin{scope}[scale = 0.8, shift = {(9.75,0)}]
\def\k{1}

\node[bvertex](v1) at (0,0){};
\node[bvertex, label=left: $k_j^2$](v2) at (0,2){};
\node[rvertex, label=below: $k_j^1$](v3) at (2,0){};
\node[rvertex](v4) at (2,2){};
\node[bvertex, label = left:$a_j$](v5) at (\k,\k){};
\node[bvertex](v6) at (\k, 2+\k){};
\node[rvertex](v7) at (2+\k, \k){};
\node[rvertex, label=above: $k_j^3$](v8) at (2+\k, 2+\k){};

\draw[edge](v1) -- (v2);
\draw[tedge](v1) -- (v3);
\draw[edge](v1) -- (v5);
\draw[tedge](v2) -- (v4);
\draw[edge](v2) -- (v6);
\draw[edge](v3) -- (v4);
\draw[edge](v3) -- (v7);
\draw[edge](v4) -- (v8);
\draw[edge](v5) -- (v6);
\draw[tedge](v6) -- (v8);
\draw[tedge](v7) -- (v5);
\draw[edge](v8) -- (v7);
\end{scope}

\begin{scope}[scale = 0.8, shift = {(14,0)}]
\def\k{1}

\node[bvertex](v1) at (0,0){};
\node[rvertex, label=left: $k_j^2$](v2) at (0,2){};
\node[bvertex, label=below: $k_j^1$](v3) at (2,0){};
\node[rvertex](v4) at (2,2){};
\node[bvertex, label = left:$a_j$](v5) at (\k,\k){};
\node[rvertex](v6) at (\k, 2+\k){};
\node[bvertex](v7) at (2+\k, \k){};
\node[rvertex, label=above: $k_j^3$](v8) at (2+\k, 2+\k){};

\draw[tedge](v1) -- (v2);
\draw[edge](v1) -- (v3);
\draw[edge](v1) -- (v5);
\draw[edge](v2) -- (v4);
\draw[edge](v2) -- (v6);
\draw[tedge](v3) -- (v4);
\draw[edge](v3) -- (v7);
\draw[edge](v4) -- (v8);
\draw[tedge](v5) -- (v6);
\draw[edge](v6) -- (v8);
\draw[edge](v7) -- (v5);
\draw[tedge](v8) -- (v7);
\end{scope}
\end{scope}
\end{tikzpicture}}
\caption{\label{f-pmc-biprad4-clause}The clause gadget in the proof of Theorem~\ref{t-pmc-biprad4} (left) and the three ways how it can be coloured with the resulting perfect matching cut.}
\end{figure}
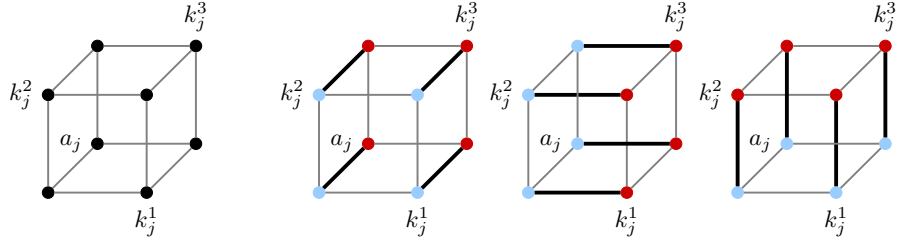

We construct a graph $G$ as follows.
For every clause $C_j$ we take a cube $K_j$ as \emph{clause gadget} with three \emph{clause vertices} $k_j^1, k_j^2$ and $k_j^3$, and one \emph{auxiliary vertex} $a_j$, see Figure~\ref{f-pmc-biprad4-clause} (left).
For every variable $x_i$ we add a \emph{variable vertex} $x_i$ together with a \emph{dummy vertex} $x_i'$, adjacent to $x_i$.
We connect $x_i$ to a clause vertex in $K_j$ if and only if $C_j$ contains $x_i$ as a variable.
We refer to the clause vertex in $K_j$ to which $x_i$ is adjacent as $k_j^{x_i}$.
Further, for every auxiliary vertex $a_j$, we add two more auxiliary vertices $b_j^1$ and $b_j^2$, together with the edges $a_jb_j^1$ and $b_j^1b_j^2$.

To bound the radius, we add a \emph{center} vertex $u$ connected to another dummy vertex $u'$.
For every occurrence of a variable $x_i$ in a clause $C_j$, we add a \emph{connector gadget} as depicted in Figure~\ref{f-pmc-biprad4-connbip} (left). This connects all the clauses and dummy vertices to the center vertex $u$.

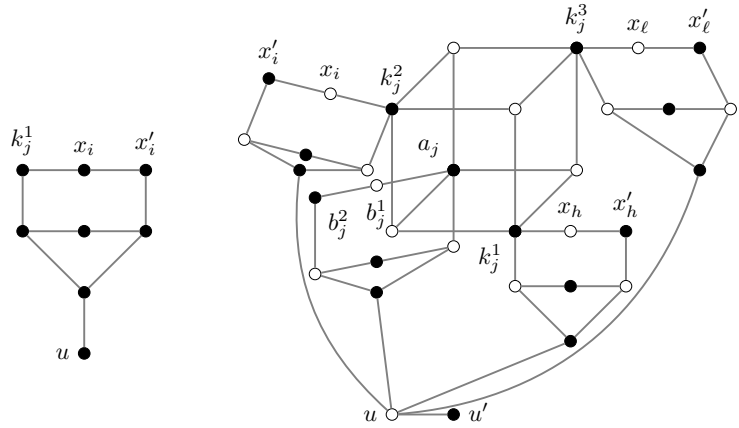
\begin{figure}
\centering
\scalebox{0.9}{
\begin{tikzpicture}
\begin{scope}[scale = 0.9]
\begin{scope}
\node[vertex, label=$k_j^1$](c1) at (0,3){};
\node[vertex, label= $x_i$](c2) at (1,3){};
\node[vertex, label = $x_i'$](c3) at (2,3){};
\node[vertex](c4) at (0,2){};
\node[vertex](c5) at (1,2){};
\node[vertex](c6) at (2,2){};
\node[vertex](c7) at (1,1){};

\node[vertex, label =left: $u$](u) at (1,0){};

\draw[edge](c1) -- (c2);
\draw[edge](c2) -- (c3);
\draw[edge](c3) -- (c6);
\draw[edge](c4) -- (c5);
\draw[edge](c5) -- (c6);
\draw[edge](c4) -- (c1);
\draw[edge](c4) -- (c7);
\draw[edge](c6) -- (c7);
\draw[edge](u) -- (c7);

\end{scope}

\begin{scope}[shift = {(6,2)}]

\def\k{1}

\node[evertex](v1) at (0,0){};
\node[vertex, label=above: $k_j^2$](v2) at (0,2){};
\node[vertex, label=below left: $k_j^1$](v3) at (2,0){};
\node[evertex](v4) at (2,2){};
\node[vertex, label = above left:$a_j$](v5) at (1,1){};
\node[evertex](v6) at (1, 3){};
\node[evertex](v7) at (3, 1){};
\node[vertex, label=above: $k_j^3$](v8) at (3, 3){};

\draw[edge](v1) -- (v2);
\draw[edge](v1) -- (v3);
\draw[edge](v1) -- (v5);
\draw[edge](v2) -- (v4);
\draw[edge](v2) -- (v6);
\draw[edge](v3) -- (v4);
\draw[edge](v3) -- (v7);
\draw[edge](v4) -- (v8);
\draw[edge](v5) -- (v6);
\draw[edge](v6) -- (v8);
\draw[edge](v7) -- (v5);
\draw[edge](v8) -- (v7);

\begin{scope}
\node[vertex](c11) at (3,3){};
\node[evertex, label=above:$x_\ell$](c12) at (4,3){};
\node[vertex, label=above:$x_\ell'$](c13) at (5,3){};
\node[evertex](c14) at (3.5,2){};
\node[vertex](c15) at (4.5,2){};
\node[evertex](c16) at (5.5,2){};
\node[vertex](c17) at (5,1){};

\draw[edge](c11) -- (c12);
\draw[edge](c12) -- (c13);
\draw[edge](c13) -- (c16);
\draw[edge](c14) -- (c15);
\draw[edge](c15) -- (c16);
\draw[edge](c14) -- (c11);
\draw[edge](c14) -- (c17);
\draw[edge](c16) -- (c17);

\end{scope}

\begin{scope}
\node[vertex](c21) at (2,0){};
\node[evertex, label=above:$x_h$](c22) at (2.9,0){};
\node[vertex, label=above:$x_h'$](c23) at (3.8,0){};
\node[evertex](c24) at (2,-0.9){};
\node[vertex](c25) at (2.9,-0.9){};
\node[evertex](c26) at (3.8,-0.9){};
\node[vertex](c27) at (2.9,-1.8){};

\draw[edge](c21) -- (c22);
\draw[edge](c22) -- (c23);
\draw[edge](c23) -- (c26);
\draw[edge](c24) -- (c25);
\draw[edge](c25) -- (c26);
\draw[edge](c24) -- (c21);
\draw[edge](c24) -- (c27);
\draw[edge](c26) -- (c27);

\end{scope}

\begin{scope}
\node[vertex, label=above:$x_i'$](c31) at (-2,2.5){};
\node[evertex, label=above:$x_i$](c32) at (-1,2.25){};
\node[vertex](c33) at (0,2){};
\node[evertex](c34) at (-2.4,1.5){};
\node[vertex](c35) at (-1.4,1.25){};
\node[evertex](c36) at (-0.4,1){};
\node[vertex](c37) at (-1.5,1){};

\draw[edge](c31) -- (c32);
\draw[edge](c32) -- (c33);
\draw[edge](c33) -- (c36);
\draw[edge](c34) -- (c35);
\draw[edge](c35) -- (c36);
\draw[edge](c34) -- (c31);
\draw[edge](c34) -- (c37);
\draw[edge](c36) -- (c37);

\end{scope}

\begin{scope}
\node[vertex, label=below right:$b_j^2$](c41) at (-1.25,0.55){};
\node[evertex, label=below:$b_j^1$](c42) at (-0.25,0.75){};
\node[vertex](c43) at (1,1){};
\node[evertex](c44) at (-1.25,-0.7){};
\node[vertex](c45) at (-0.25,-0.5){};
\node[evertex](c46) at (1,-0.25){};
\node[vertex](c47) at (-0.25,-1){};

\draw[edge](c41) -- (c42);
\draw[edge](c42) -- (c43);
\draw[edge](c43) -- (c46);
\draw[edge](c44) -- (c45);
\draw[edge](c45) -- (c46);
\draw[edge](c44) -- (c41);
\draw[edge](c44) -- (c47);
\draw[edge](c46) -- (c47);

\end{scope}

\node[evertex, label= {left:\strut{$u$}}](u) at (0,-3){};
\node[vertex, label= {right:\strut{$u'$}}](up) at (1,-3){};

\draw[edge](u) to [bend right= 34] (c17);
\draw[edge](u) -- (c27);
\draw[edge](u) to [bend left= 27] (c37);
\draw[edge](u) -- (c47);
\draw[edge](u) -- (up);

\end{scope}
\end{scope}
\end{tikzpicture}}
\caption{\label{f-pmc-biprad4-connbip}The connector gadget in the proof of Theorem~\ref{t-pmc-biprad4} (left) and the bipartition of the clause gadget of $C_j = (x_i,x_h, x_\ell)$ with the corresponding variable vertices and connector gadgets (right).}
\end{figure}

Observe first that $G$ is bipartite. One partition class consists of the clause vertices, the dummy vertices, the auxiliary vertices $a_j$ and $b_j^2$, for $j \in \{1,\dots,m\}$, and the vertices in the connector gadget as depicted in Figure~\ref{f-pmc-biprad4-connbip} (right).

To see that $G$ has radius~$4$, note first that all vertices in the connector gadgets are within distance~$4$ from $u$.
Note further that the clause vertices and the auxiliary vertex $a_j$ are at distance~$3$ from $u$.
Since every other vertex in a clause gadget is at distance $1$ from a clause vertex or $a_j$, for $ j \in \{1,\dots, m\}$, $G$ has radius~$4$.

By Observation~\ref{o-cut-colouring} we know that $G$ has a perfect matching cut if and only if $G$ has a perfect red-blue colouring.
Thus, it remains to show that $G$ has a perfect red-blue colouring if and only if $(X,\mathcal{C})$ is a \yes-instance of \naesat.

\begin{claim}\label{c-pmc-biprad4-cgmono}
In every perfect red-blue colouring of $G$, every clause gadget has one of the three colourings given in Figure~\ref{f-pmc-biprad4-clause} (right).
\end{claim}
\begin{claimproof}
Note first that there are vertices in the clause gadget which have neighbours only inside the clause gadgets.
Thus, in every perfect red-blue colouring of $G$, no clause gadget can be monochromatic. Note further that the given colourings are the only possible bichromatic colourings of the cube.
\end{claimproof}

\begin{claim}\label{c-pmc-biprad4-vcmono}
In every perfect red-blue colouring of $G$, a variable vertex $x_i$, $ i \in \{1,\dots, n\}$, and its adjacent clause vertices $k_j^{x_i}$, $j \in \{1,\dots, m\}$ have the same colour. The same also holds for $a_j$ and $b_j^1$.
\end{claim}
\begin{claimproof}
Suppose there is a red variable vertex $x_i$ with an adjacent blue clause vertex $k_j^{x_i}$.
Then, all other neighbours of $k_j^{x_i}$ are blue.
Thus, the whole clause gadget is coloured blue, a contradiction to Claim~\ref{c-pmc-biprad4-cgmono}.
\end{claimproof}

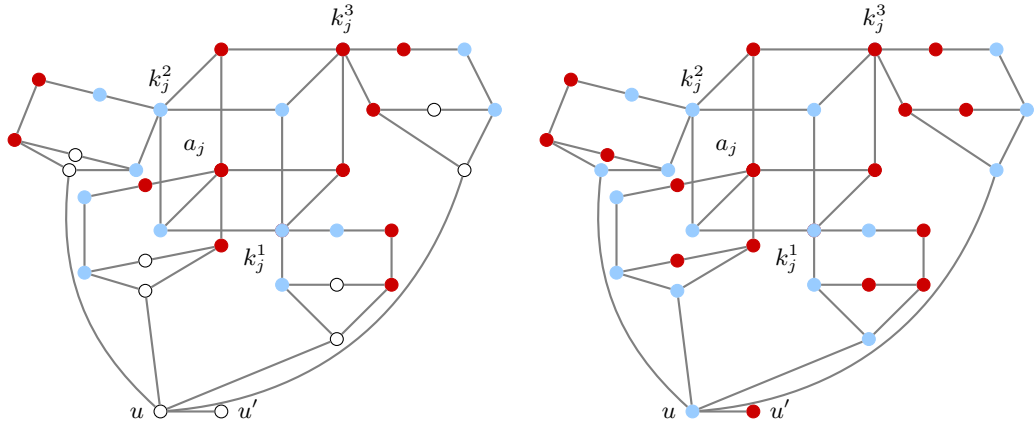
\begin{figure}
\centering
\begin{tikzpicture}

\begin{scope}[shift = {(0,0)}, scale = 0.8]

\def\k{1}

\node[bvertex](v1) at (0,0){};
\node[bvertex, label=above: {\small $k_j^2$}](v2) at (0,2){};
\node[rvertex, label=below left: {\small $k_j^1$}](v3) at (2,0){};
\node[bvertex](v4) at (2,2){};
\node[rvertex, label = above left: {\small $a_j$}](v5) at (1,1){};
\node[rvertex](v6) at (1, 3){};
\node[rvertex](v7) at (3, 1){};
\node[bvertex, label=above: {\small $k_j^3$}](v8) at (3, 3){};

\draw[edge](v1) -- (v2);
\draw[edge](v1) -- (v3);
\draw[edge](v1) -- (v5);
\draw[edge](v2) -- (v4);
\draw[edge](v2) -- (v6);
\draw[edge](v3) -- (v4);
\draw[edge](v3) -- (v7);
\draw[edge](v4) -- (v8);
\draw[edge](v5) -- (v6);
\draw[edge](v6) -- (v8);
\draw[edge](v7) -- (v5);
\draw[edge](v8) -- (v7);

\begin{scope}
\node[rvertex](c11) at (3,3){};
\node[rvertex](c12) at (4,3){};
\node[bvertex](c13) at (5,3){};
\node[rvertex](c14) at (3.5,2){};
\node[evertex](c15) at (4.5,2){};
\node[bvertex](c16) at (5.5,2){};
\node[evertex](c17) at (5,1){};

\draw[edge](c11) -- (c12);
\draw[edge](c12) -- (c13);
\draw[edge](c13) -- (c16);
\draw[edge](c14) -- (c15);
\draw[edge](c15) -- (c16);
\draw[edge](c14) -- (c11);
\draw[edge](c14) -- (c17);
\draw[edge](c16) -- (c17);

\end{scope}

\begin{scope}
\node[bvertex](c21) at (2,0){};
\node[bvertex](c22) at (2.9,0){};
\node[rvertex](c23) at (3.8,0){};
\node[bvertex](c24) at (2,-0.9){};
\node[evertex](c25) at (2.9,-0.9){};
\node[rvertex](c26) at (3.8,-0.9){};
\node[evertex](c27) at (2.9,-1.8){};

\draw[edge](c21) -- (c22);
\draw[edge](c22) -- (c23);
\draw[edge](c23) -- (c26);
\draw[edge](c24) -- (c25);
\draw[edge](c25) -- (c26);
\draw[edge](c24) -- (c21);
\draw[edge](c24) -- (c27);
\draw[edge](c26) -- (c27);

\end{scope}

\begin{scope}
\node[rvertex](c31) at (-2,2.5){};
\node[bvertex](c32) at (-1,2.25){};
\node[bvertex](c33) at (0,2){};
\node[rvertex](c34) at (-2.4,1.5){};
\node[evertex](c35) at (-1.4,1.25){};
\node[bvertex](c36) at (-0.4,1){};
\node[evertex](c37) at (-1.5,1){};

\draw[edge](c31) -- (c32);
\draw[edge](c32) -- (c33);
\draw[edge](c33) -- (c36);
\draw[edge](c34) -- (c35);
\draw[edge](c35) -- (c36);
\draw[edge](c34) -- (c31);
\draw[edge](c34) -- (c37);
\draw[edge](c36) -- (c37);

\end{scope}

\begin{scope}
\node[bvertex](c41) at (-1.25,0.55){};
\node[rvertex](c42) at (-0.25,0.75){};
\node[rvertex](c43) at (1,1){};
\node[bvertex](c44) at (-1.25,-0.7){};
\node[evertex](c45) at (-0.25,-0.5){};
\node[rvertex](c46) at (1,-0.25){};
\node[evertex](c47) at (-0.25,-1){};

\draw[edge](c41) -- (c42);
\draw[edge](c42) -- (c43);
\draw[edge](c43) -- (c46);
\draw[edge](c44) -- (c45);
\draw[edge](c45) -- (c46);
\draw[edge](c44) -- (c41);
\draw[edge](c44) -- (c47);
\draw[edge](c46) -- (c47);

\end{scope}

\node[evertex, label= {left:\strut{\small $u$}}](u) at (0,-3){};
\node[evertex, label= {right:\strut{\small $u'$}}](up) at (1,-3){};

\draw[edge](u) to [bend right= 34] (c17);
\draw[edge](u) -- (c27);
\draw[edge](u) to [bend left= 27] (c37);
\draw[edge](u) -- (c47);
\draw[edge](u) -- (up);

\end{scope}

\begin{scope}[shift = {(7,0)}, scale = 0.8]

\def\k{1}

\node[bvertex](v1) at (0,0){};
\node[bvertex, label=above: {\small $k_j^2$}](v2) at (0,2){};
\node[rvertex, label=below left: {\small $k_j^1$}](v3) at (2,0){};
\node[bvertex](v4) at (2,2){};
\node[rvertex, label = above left: {\small $a_j$}](v5) at (1,1){};
\node[rvertex](v6) at (1, 3){};
\node[rvertex](v7) at (3, 1){};
\node[bvertex, label=above: {\small $k_j^3$}](v8) at (3, 3){};

\draw[edge](v1) -- (v2);
\draw[edge](v1) -- (v3);
\draw[edge](v1) -- (v5);
\draw[edge](v2) -- (v4);
\draw[edge](v2) -- (v6);
\draw[edge](v3) -- (v4);
\draw[edge](v3) -- (v7);
\draw[edge](v4) -- (v8);
\draw[edge](v5) -- (v6);
\draw[edge](v6) -- (v8);
\draw[edge](v7) -- (v5);
\draw[edge](v8) -- (v7);

\begin{scope}
\node[rvertex](c11) at (3,3){};
\node[rvertex](c12) at (4,3){};
\node[bvertex](c13) at (5,3){};
\node[rvertex](c14) at (3.5,2){};
\node[rvertex](c15) at (4.5,2){};
\node[bvertex](c16) at (5.5,2){};
\node[bvertex](c17) at (5,1){};

\draw[edge](c11) -- (c12);
\draw[edge](c12) -- (c13);
\draw[edge](c13) -- (c16);
\draw[edge](c14) -- (c15);
\draw[edge](c15) -- (c16);
\draw[edge](c14) -- (c11);
\draw[edge](c14) -- (c17);
\draw[edge](c16) -- (c17);

\end{scope}

\begin{scope}
\node[bvertex](c21) at (2,0){};
\node[bvertex](c22) at (2.9,0){};
\node[rvertex](c23) at (3.8,0){};
\node[bvertex](c24) at (2,-0.9){};
\node[rvertex](c25) at (2.9,-0.9){};
\node[rvertex](c26) at (3.8,-0.9){};
\node[bvertex](c27) at (2.9,-1.8){};

\draw[edge](c21) -- (c22);
\draw[edge](c22) -- (c23);
\draw[edge](c23) -- (c26);
\draw[edge](c24) -- (c25);
\draw[edge](c25) -- (c26);
\draw[edge](c24) -- (c21);
\draw[edge](c24) -- (c27);
\draw[edge](c26) -- (c27);

\end{scope}

\begin{scope}
\node[rvertex](c31) at (-2,2.5){};
\node[bvertex](c32) at (-1,2.25){};
\node[bvertex](c33) at (0,2){};
\node[rvertex](c34) at (-2.4,1.5){};
\node[rvertex](c35) at (-1.4,1.25){};
\node[bvertex](c36) at (-0.4,1){};
\node[bvertex](c37) at (-1.5,1){};

\draw[edge](c31) -- (c32);
\draw[edge](c32) -- (c33);
\draw[edge](c33) -- (c36);
\draw[edge](c34) -- (c35);
\draw[edge](c35) -- (c36);
\draw[edge](c34) -- (c31);
\draw[edge](c34) -- (c37);
\draw[edge](c36) -- (c37);

\end{scope}

\begin{scope}
\node[bvertex](c41) at (-1.25,0.55){};
\node[rvertex](c42) at (-0.25,0.75){};
\node[rvertex](c43) at (1,1){};
\node[bvertex](c44) at (-1.25,-0.7){};
\node[rvertex](c45) at (-0.25,-0.5){};
\node[rvertex](c46) at (1,-0.25){};
\node[bvertex](c47) at (-0.25,-1){};

\draw[edge](c41) -- (c42);
\draw[edge](c42) -- (c43);
\draw[edge](c43) -- (c46);
\draw[edge](c44) -- (c45);
\draw[edge](c45) -- (c46);
\draw[edge](c44) -- (c41);
\draw[edge](c44) -- (c47);
\draw[edge](c46) -- (c47);

\end{scope}

\node[bvertex, label= {left:\strut{\small $u$}}](u) at (0,-3){};
\node[rvertex, label= {right:\strut{\small $u'$}}](up) at (1,-3){};

\draw[edge](u) to [bend right= 34] (c17);
\draw[edge](u) -- (c27);
\draw[edge](u) to [bend left= 27] (c37);
\draw[edge](u) -- (c47);
\draw[edge](u) -- (up);

\end{scope}
\end{tikzpicture}
\caption{\label{f-pmc-biprad4-col}The partial colouring resulting from the truth assignment (left) and one of the two ways to complete it (right).}
\end{figure}

\noindent
Assume first that $(X,\mathcal{C})$ is a \yes-instance of \naesat.
For every variable which is set to true, we colour the corresponding variable vertex blue and for every variable set to false, we colour the corresponding variable vertex red.
Note that since $(X,\mathcal{C})$ is a \yes-instance of \naesat, for each clause at least one variable vertex is blue and at least one is red.
Note further that there is a unique way to extend this colouring to a valid colouring of the clause gadgets. 
We extend this colouring to the variable vertices using Claim~\ref{c-pmc-biprad4-vcmono}.
We propagate the colouring further in the only possible way until everything expect $u$, $u'$ and two vertices per connector gadget is coloured, see Figure~\ref{f-pmc-biprad4-col} (left).
We colour the center vertex $u$ blue and extend the colouring in the only possible way, see Figure~\ref{f-pmc-biprad4-col} (right).
This leads to a perfect red-blue colouring of $G$.

For the other direction assume that $G$ admits a perfect red-blue colouring.
We set a variable to true if its corresponding variable vertex is coloured blue and red otherwise.
By Claim~\ref{c-pmc-biprad4-cgmono} we know that for every perfect red-blue colouring in every clause gadget at least one clause vertex is coloured red and one blue. Further, by Claim~\ref{c-pmc-biprad4-vcmono}, we know that a variable vertex has the same colour as its adjacent clause vertices.
Thus, the assignment is consistent and in every clause there is at least one variable set to true and one to false.
Therefore, $(X,\mathcal{C})$ is a \yes-instance of \naesat.
\end{proof}

\noindent
By Observation~\ref{o-c-all-kr-mono}, complete bipartite graphs have to be monochromatic if each partition class is large enough. 
For the following results we modify the constructions in the case of non-bipartite graphs with bounded radius and diameter. The general idea is the following. 
We transform all cliques into bicliques such that a vertex $v$ in the original graph $G$ corresponds to two new vertices $v_1, v_2$ in the new graph $G'$. 
Let $V_1, V_2$ be the partition classes of $G'$. In $G'$, $v_1\in V_1$ and $v_2\in V_2$. The neighbourhood of $v_1$ is the set of vertices in $V_2$ who correspond to the neighbours of $v$ in $G$. Similarly, the neighbours of $v_2$ are those corresponding to neighbours of $v$ in $V_1$.
Therefore, the edges between variable gadgets and clause gadgets exist twice, leading to two copies of $G$ where cliques are replaced by independent sets.
The colouring properties are preserved due to the connection of the two copies. This construction assures that the radius and diameter of the graph only increase by~$1$.

In the following we apply this construction to a construction from~\cite{LMPS24}.
Note that the fact that $d \geq 2$ is important in the proof.

\begin{theorem}\label{t-dc-dc-biprad3}
For $d\geq 2$, \dc{} is \NP-complete for bipartite graphs of radius~$3$ and diameter~$4$.
\end{theorem}

\begin{proof}
We reduce from \naesat. Let $(X,{\cal C})$ be an instance of this problem, where $X= \{x_1,x_2,\dots,x_n\}$ and ${\cal C} = \{C_1, C_2, \dots , C_m\}$. By Theorem~\ref{t-np-none-naesat}, we may assume that each literal occurs only positively and in at most three different clauses, and each $C_i$ contains either two or three literals.

From $(X,{\cal C})$ we construct a graph $G$ as follows, see Figure~\ref{f-dc-biprad3} for an example.
We introduce a complete bipartite graph $K $ with partition classes $K_a = \{C_{1_1}^a,  \dots, C_{1_d}^a, \dots, C_{m_1}^a, \dots,  C_{m_d}^a\}$ and $K_b = \{C_{1_1}^b, \dots, C_{1_d}^b, \dots, C_{m_1}^b, \dots,  C_{m_d}^b\}$ and a complete bipartite graph $K'$ with partition classes $K'_a = \{C_{1_1}'^a, \dots, C_{1_d}'^a, \dots, C_{m_1}'^a, \dots,  C_{m_d}'^a\}$ and $K'_b = \{C_{1_1}'^b, \dots, C_{1_d}'^b, \dots, C_{m_1}'^b, \dots,  C_{m_d}'^b\}$.
For every variable $x_i$ we add a complete bipartite graph $X_i$ with partition classes $X_i^a$ and $X_i^b$ of size~$3d$ each.
For each clause $C_j = \{x_i, x_k,  x_\ell\}$ consisting of three literals we choose 
\begin{itemize}
\item $d-1$ vertices of $X_i^a$ and make them complete to $C_{j_1}^b, \dots, C_{j_d}^b$ and $C_{j_1}'^b, \dots, C_{j_d}'^b$;
\item $d-1$ vertices of $X_i^b$ and make them complete to $C_{j_1}^a, \dots, C_{j_d}^a$ and $C_{j_1}'^a, \dots, C_{j_d}'^a$;
\item one vertex of $X_k^a$ and add edges to $C_{j_1}^b, \dots, C_{j_d}^b$ and $C_{j_1}'^b, \dots, C_{j_d}'^b$;
\item one vertex of $X_k^b$ and add edges to $C_{j_1}^a, \dots, C_{j_d}^a$ and $C_{j_1}'^a, \dots, C_{j_d}'^a$;
\item one vertex of $X_\ell^a$ and add edges to  $C_{j_1}^b, \dots, C_{j_d}^b$ and $C_{j_1}'^b, \dots, C_{j_d}'^b$; and
\item one vertex of $X_\ell^b$ and add edges to  $C_{j_1}^a, \dots, C_{j_d}^a$ and $C_{j_1}'^a, \dots, C_{j_d}'^a$.
\end{itemize}
For each clause $C_j = \{x_i, x_k\}$ consisting of two literals we choose
\begin{itemize}
\item $d$ vertices of $X_i^a$ and make them complete to $C_{j_1}^b, \dots, C_{j_d}^b$ and $C_{j_1}'^b, \dots, C_{j_d}'^b$;
\item $d$ vertices of $X_i^b$ and make them complete to $C_{j_1}^a, \dots, C_{j_d}^a$ and $C_{j_1}'^a, \dots, C_{j_d}'^a$;
\item one vertex of $X_k^a$ and add edges to $C_{j_1}^b, \dots, C_{j_d}^b$ and $C_{j_1}'^b, \dots, C_{j_d}'^b$; and
\item one vertex of $X_k^b$ and add edges to $C_{j_1}^a, \dots, C_{j_d}^a$ and $C_{j_1}'^a, \dots, C_{j_d}'^a$.
\end{itemize}

\begin{figure}
\centering
\scalebox{0.9}{
\begin{tikzpicture}

	\tikzstyle{twovbox}=[draw,lightblue, thick, rounded corners, minimum width = 1cm, minimum height = 0.4cm]
	\tikzstyle{threevbox}=[draw,lightblue,thick, rounded corners, minimum width = 1.4cm, minimum height = 0.4cm]
	
	\begin{scope}[scale = 0.9]
	\begin{scope}[shift = {(0,-0.5)}]
	
	\draw (0,5) rectangle (4,6);	
	\draw (0,6) rectangle (4,7);
	\node[](ka) at (-0.5, 5.5){$K_a$};	
	\node[](kb) at (-0.5, 6.5){$K_b$};	
	
	\draw (7,5) rectangle (11,6);	
	\draw (7,6) rectangle (11,7);
	\node[](kap) at (11.5, 5.5){$K_a'$};	
	\node[](kbp) at (11.5, 6.5){$K_b'$};

	\node[vertex, label={[left, label distance = 0.1]:$u_b$}] (ua) at (3.5,6.5){};
	\node[vertex, label={[left, label distance = 0.1]:$u_a$}] (ub) at (3.5,5.5){};
	\node[vertex, label={[right, label distance = 0.1]: \strut $u_b'$}] (uap) at (7.5,6.5){};
	\node[vertex, label={[right, label distance = 0.1]:\strut $u_a'$}] (ubp) at (7.5,5.5){};
	
	\draw[edge] (ua) -- (ubp);
	\draw[edge] (uap) -- (ub);
	
	\node[vertex](c1a) at (0.5,5.5){};
	\node[vertex](c2a) at (1,  5.5){};
	\node[vertex](c3a) at (1.5,5.5){};
	\node[threevbox](ca) at (1,5.5){};
	
	\node[vertex](c1b) at (1.5,  6.5){};
	\node[vertex](c2b) at (2,6.5){};
	\node[vertex](c3b) at (2.5,  6.5){};
	\node[threevbox](cb) at (2,6.5){};
	
	\node[vertex](c1ap) at   (9.5,  5.5){};
	\node[vertex](c2ap) at   (10,5.5){};
	\node[vertex](c3ap) at   (10.5,  5.5){};
	\node[threevbox](cap) at (10,5.5){};
	
	\node[vertex](c1bp) at   (8.5, 6.5){};
	\node[vertex](c2bp) at   (9,  6.5){};
	\node[vertex](c3bp) at   (9.5,6.5){};	
	\node[threevbox](cbp) at (9,  6.5){};
	
	\end{scope}

		\draw (0,0) rectangle (3,1);	
	\draw (0,1) rectangle (3,2);	
	\node[](xia) at (-0.5, 1.5){$X_i^a$};	
	\node[](xib) at (-0.5, 0.5){$X_i^b$};

\begin{scope}[shift = {(0,0)}]	
	\draw (4.5,0) rectangle (6.5,1);	
	\draw (4.5,1) rectangle (6.5,2);	
	\node[](xka) at (4, 1.5){$X_k^a$};	
	\node[](xkb) at (4, 0.5){$X_k^b$};
		\node[vertex](xkb) at (5.5,0.5){};
	\node[vertex](xka) at (5.5,1.5){};	
	\end{scope}

	\draw (8,0) rectangle (11,1);	
	\draw (8,1) rectangle (11,2);	
	\node[](xla) at (11.5, 1.5){$X_\ell^a$};	
	\node[](xlb) at (11.5, 0.5){$X_\ell^b$};

	\node[vertex](xi1b) at (0.75,0.5){};
	\node[vertex](xi2b) at (1.25,  0.5){};
	\node[twovbox](xib) at (1,0.5){};
	
	\node[vertex](xi1a) at (2,  1.5){};
	\node[vertex](xi2a) at (2.5,1.5){};
	\node[twovbox](xia) at (2.25,1.5){};

	\node[vertex](xlb) at (10, 0.5){};
	\node[vertex](xla) at (8.75,1.5){};	
	
	\draw[edge](ca.south) -- (xib);
	\draw[edge](ca.south) -- (xkb);
	\draw[edge](ca.south) to [bend right=17] (xlb);
	\draw[edge](cb.south) -- (xia);
	\draw[edge](cb.south) -- (xka);
	\draw[edge](cb.south) to [bend right=18] (xla);
	
	\draw[edge](cap.south) to [bend left=15] (xib.east);
	\draw[edge](cap.south) -- (xkb);
	\draw[edge](cap.south) -- (xlb);
	\draw[edge](cbp.south) to [bend left=18] (xia.east);
	\draw[edge](cbp.south) -- (xka);
	\draw[edge](cbp.south) -- (xla);
	
	

\end{scope}
	
\end{tikzpicture}}
\caption{\label{f-dc-biprad3} Representation of a clause $\{x_i, x_k, x_\ell\}$, in the case of $d=3$. For readability, the edges between the partition classes are not shown and edges incident to a blue box stand for edges to all vertices in the blue box.}  
\end{figure}
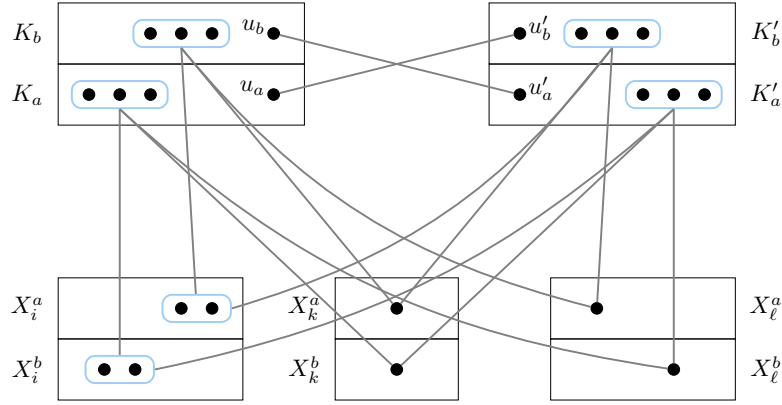
\noindent
In all cases, ensure that every vertex in every $X_i$ is associated to at most one clause. 
We call those vertices of $X_i$ which are not selected to represent an occurrence of $x_i$ \emph{auxiliary vertices}.
 For each of these auxiliary vertices $x$, we add one vertex to $K_b$ and one to $K'_b$ if $x \in X_i^a$ and make them adjacent to $x$. If $x \in X_i^b$ we add one vertex to $K_a$ and one to $K'_a$ and make them adjacent to $x$. To complete the construction, we add vertices $u_a$ to $K_a$, $u_b$ to $K_b$, $u'_a$ to $K'_a$ and  $u'_b$ to $K'_b$. Finally, we add the edges $u_au'_b$ and $u_bu'_a$.

\smallskip
\noindent
Note first that $G$ is indeed bipartite since all vertices in $X_i^a$ and $K_a$ are only adjacent to vertices in $X_i^b$ and $K_b$ and vice versa. Further, $X_i^a$, $K_a$, $X_i^b$ and $K_b$ are all independent sets.
We show that $G$ has radius $3$ and diameter $4$. 
Consider the vertex $u_a \in V(K)$ connecting $K$ and $K'$. Since every vertex in the variable gadgets has a neighbour in $K_a$ or $K_b$, the distance from $u$ to any vertex of $G$ is at most~$3$ and thus, $G$ has radius~$3$.
Let $v \in V(X_i)$ and $v' \in V(X_j)$ for $i, j \in \{1, \dots, n\}$. Then, there is a path from $v$ to $v'$ of length at most~$4$ using their neighbours in $K$ which are at distance at most~$2$.
Similarly, for $v \in V(K)$ and $v' \in V(K')$, $v$ is adjacent to either $u_a$ or $u_b$ who are neighbours of $u_a'$ respectively $u_b'$. The distance from $v'$ to $u_a'$ and $u_b'$ is at most~$2$, resulting in a path of length at most~$4$.
This shows that $G$ has diameter~$4$.

\smallskip
\noindent
We claim that $(X,\cal C)$ is a yes-instance of \naesat{} if and only if $G$ has a $d$-cut.

\smallskip
\noindent
First suppose $(X, \cal C)$ is a yes-instance of \naesat. Then $X$ has a truth assignment $\phi$ that sets, in each clause $C_i$, at least one literal to true and at least one literal to false.
We colour a variable gadget $X_i$ red if the corresponding variable is set to true and blue otherwise.
We colour the vertices in $K$ red and those in $K'$ blue.
Consider a vertex $v$ in~$X_i$, say $v$ is coloured red. Since $v$ has the same colour as its variable gadget $X_i$ it has at most $d$ neighbours of the other colour which are either in $K$ or in $K'$.
Now consider a vertex~$C_{i_j}$ in~$K$, which is coloured red. As each clause contains two or three literals and $\phi$ sets at least one literal true, we find that $C_{i_j}$ is adjacent to at most $d$ blue vertices in the variable cliques.
The other cases follow by symmetry. Hence, we obtain a red-blue $d$-colouring of $G$. By Observation~\ref{o-cut-colouring} we get that $G$ has a $d$-cut.

Now suppose $G$ has a $d$-cut. By Observation~\ref{o-cut-colouring} this implies that $G$ has a red-blue $d$-colouring $c$. We may assume without loss of generality that $m \geq 2$, which implies that the partition classes of $K$ and $K'$ have size at least~$2d+1$ and thus,  $K$ and $K'$ are monochromatic.
Say $c$ colours every vertex of $K$ red. For a contradiction, assume $c$ colours every vertex of~$K'$ red as well. Then, in each variable gadget there is a vertex with $2d$~red neighbours so this vertex has to be coloured red. Since the partition classes of the variable gadgets have size~$3d$, the variable gadgets are monochromatic and thus, the whole graph is coloured red, a contradiction. We conclude that $c$ must colour every vertex of $K'$ blue.

Assume for another contradiction that there exists a clause $C_i$, for $i \in \{1,\dots, m\}$ such that all its corresponding variable gadgets are coloured blue. 
Then $C_{i_j}^a$ in $K_a$, for $j \in \{1,\dots, d\}$, which is coloured red,  has $d+1$ blue neighbours, a contradiction.
By symmetry,  we get the same contradiction for $C_{i_j}^b$ in~$K_b$ and $C_{i_j}'^a$ and $C_{i_j}'^b$ in~$K'$.
Hence, setting $x_i$ to true if $X_i$ is red in $G$  and to false otherwise gives us the desired truth assignment for $X$.
\end{proof}

 \begin{figure}[ht]
 \begin{center}
\scalebox{0.9}{
\begin{tikzpicture}
\begin{scope}[scale = 0.9]

\foreach \i in {0,..., 5}{
	\node[vertex, label= above: $x_{\i}$](x\i) at (2*\i, 3.5){};
}
\foreach \i in {0,..., 5}{
	\node[vertex, label= below: $y_{\i}$](y\i) at (2*\i, -2.5){};
}


\node[vertex, label={[label distance = -1mm]above left:$x_0^{S_1}$}](c11) at (0,1){};
\node[vertex, label={[label distance = -1mm]above left:$x_1^{S_1}$}](c12) at (1,1){};
\node[vertex, label={[label distance = -0.7mm]above:$x_3^{S_1}$}](c13) at (2,1){};

\node[vertex, label={[label distance = 1mm]above  :$x_1^{S_1}$}](c21) at (4,1){};
\node[vertex,  label={[label distance = 1mm]above:$x_3^{S_1}$}](c22) at (5,1){};
\node[vertex, label={[label distance = 1mm]above:$x_4^{S_1}$}](c23) at (6,1){};

\node[vertex, label={[label distance = -0.7mm]above :$x_2^{S_1}$}](c31) at (8,1){};
\node[vertex, label={[label distance = -1mm]above right:$x_4^{S_1}$}](c32) at (9,1){};
\node[vertex, label={[label distance = -1mm]above right:$x_5^{S_1}$}](c33) at (10,1){};

\node[vertex, label={[label distance = -1mm]below left:$x_0^{S_2}$}](d11) at (0,0){};
\node[vertex, label={[label distance = -1mm]below left:$x_1^{S_2}$}](d12) at (1,0){};
\node[vertex, label={[label distance = -0.7mm]below :$x_3^{S_2}$}](d13) at (2,0){};

\node[vertex, label={[label distance = 1mm]below:$x_1^{S_2}$}](d21) at (4,0){};
\node[vertex,  label={[label distance = 1mm]below :$x_3^{S_2}$}](d22) at (5,0){};
\node[vertex, label={[label distance = 1mm]below :$x_4^{S_2}$}](d23) at (6,0){};

\node[vertex, label={[label distance = -0.7mm]below :$x_2^{S_2}$}](d31) at (8,0){};
\node[vertex,  label={[label distance = -1mm]below right:$x_4^{S_2}$}](d32) at (9,0){};
\node[vertex, label={[label distance = -1mm]below right:$x_5^{S_2}$}](d33) at (10,0){};

\foreach \i in {1,2,3}{
	\foreach \j in {1,2,3}{
		\foreach \k in {1,2, 3}{
			\draw[edge](c\i\j)--(d\i\k);
		}	
	}
}

%
%
%
%
%

\draw[edge](c21) to (x1);
\draw[edge](c22) to (x3);
\draw[edge](c23)--(x4);

\draw[tredge](c11)--(x0);
\draw[tredge] (c12)to (x1);
\draw[tredge](c13)--(x3);

\draw[tredge](c31)--(x2);
\draw[tredge](c32) to (x4);
\draw[tredge](c33)--(x5);

\draw[edge](d21) to [bend left = 0](y1);
\draw[edge](d22) to (y3);
\draw[edge](d23) to [bend right = 0](y4);

\draw[tredge](d11) to (y0); 
\draw[tredge](d12) to (y1); 
\draw[tredge](d13) to [bend left = 0](y3);

\draw[tredge](d31) to [bend left = 0](y2);
\draw[tredge](d32) to (y4); 
\draw[tredge](d33) to (y5); 

\draw[black] (-1,3) rectangle (11,4.5);
\draw[black] (-1,-2) rectangle (11,-3.5);

\node[] (t) at (-1.5, 3.75) {$X^1$};
\node[] (t) at (-1.5, -2.75) {$X^2$};

\end{scope}
\end{tikzpicture}}
 \caption{The graph $G$ for variables $X = \set{x_1,\dots, x_6}$ and sets $\mathcal{S} = \set{\set{x_1,x_2,x_4}, \set{x_2, x_4, x_5}, \set{x_3, x_5, x_6}}$. 
  The set $\mathcal{S'} =\set{\set{x_1,x_2,x_4}, \set{x_3, x_5, x_6}}$ is an exact $3$-cover of $X$. The thick red edges in the graph show the corresponding matching cut of size~$6q=12$. $X^1$ and $X^2$ are the partition classes of a complete bipartite graph. For readability the edges between $X^1$ and $X^2$ are not represented. }\label{f-np-mmc-biprad}
 \end{center}
 \end{figure}
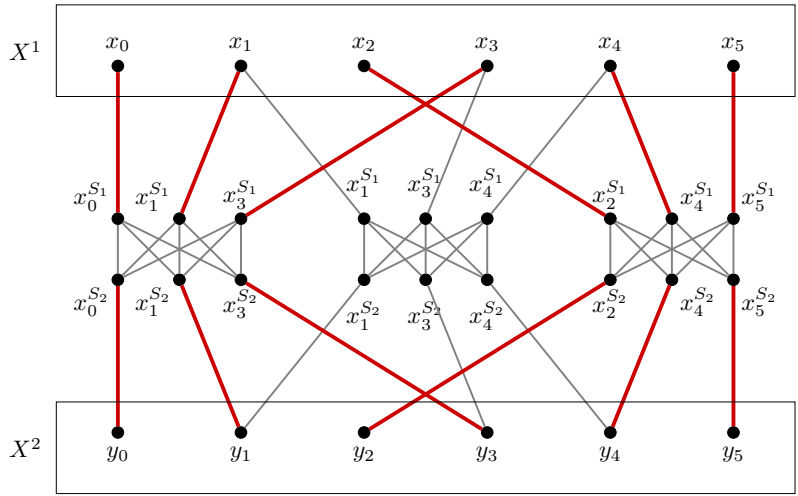

\noindent
In~\cite{LPR24} it has been shown that \mmc\ is \NP-hard for graphs of radius at most~$2$ and diameter at most~$3$.
Using the same method as before we can show a similar result for bipartite graphs.

The proof reduces from {\sc Exact $3$-Cover}.
We call a collection $\mathcal{C}$ of $3$-element subsets of~$X$ an \emph{exact $3$-cover} of $X$, if every $x \in X$ is contained in exactly one $3$-element subset of $\mathcal{C}$. In the problem {\sc Exact $3$-Cover} we are given as input a set $X=\{x_1,\ldots,x_{3\ell}\}$ of size $3\ell$ for some integer $\ell \geq 1$ and a collection ${\cal S}=\{S_1,\ldots,S_k\}$, where $S_i \subseteq X$ and $|S_i|= 3$, for all $i \in \{1,...,k\}$.
The question is whether $\mathcal{S}$ contains an exact $3$-cover of $X$. This problem is well known to be \NP-complete, see~\cite{Ka72}. 

\begin{theorem}\label{t-np-mmc-biprad}
\mmc\ is \NP-hard for bipartite graphs of radius at most~$3$ and diameter at most~$4$. 
\end{theorem}

\begin{proof}
Let $(X,\mathcal{S})$ be an instance of {\sc Exact $3$-Cover} where $X=\{x_1,\ldots,x_{3q}\}$ and ${\cal S}=\{S_1,\ldots,S_k\}$, such that each $S_i$ contains exactly three elements of $X$.
From $(X,{\cal S})$ we construct a graph $G$ as follows. 

We first define a complete bipartite graph $K_{3q,3q}$ called $K_X$ with partition classes $X^1 = \{x_1,\ldots,x_{3q}\}$ and $X^2 = \{ y_1, \dots, y_{3q}\}$. For each $S\in {\cal S}$, we do as follows. 
Let $S=\{x_h,x_i,x_j\}$. We add a complete bipartite graph $K_{3,3}$ called $K_S$ with partition classes $S^1 = \{x^{S^1}_h, x^{S^1}_i, x^{S^1}_j\}$ and $S^2 = \{x^{S^2}_h, x^{S^2}_i, x^{S^2}_j\}$. We add an edge between a vertex $x_i\in X^1$ and a vertex $u \in S^1$ if and only if $u=x_i^{S^1}$ for some $S\in {\cal S}$.
Similarly, we add an edge between a vertex $y_i\in X^2$ and a vertex $u \in S^2$ if and only if $u=x_i^{S^2}$ for some $S\in {\cal S}$. 
This completes the construction of $G$. See Figure~\ref{f-np-mmc-biprad} for an example.

First note that this graph is indeed bipartite. The partition classes are $X^1 \cup \bigcup_{1 \leq i \leq k} S_i^2$ and $X^2 \cup \bigcup_{1 \leq i \leq k} S_i^1$.
To see that $G$ has radius~$3$, pick a vertex $v$ in $X^1$. All vertices in $K_X$ have distance at most~$2$ to $v$.
Further, all vertices in some $K_S$, for $S \in \mathcal{S}$, are at distance at most $3$ to~$v$: if $u \in S^2$ go from $v \in X^1$ to $X^2$, then to $u$. Similarly, if $u \in S^1$ go from $v \in X^1$ to $X_2$, then to $S^2$ and then to $u$.

To see that $G$ has also diameter~$4$, it remains to show that the distance between $v \in S^1$ to a vertex in $K_{S'}, S' \in \mathcal{S}$ is at most~$4$. Note that the case where $v \in S^2$ follows by symmetry.  Let $v \in S^1$, for some $S \in \mathcal{S}$. First, note that all vertices in $K_S$ have distance at most $2$ to $v$.
Let $S' \in \mathcal{S}$ be different from $S$ and let $u \in S'^2$. Then there is a path of length $3$ from $v$ to $u$ using first the neighbour of $v$ in $X^1$ and then the neighbour of $u$ in~$X^2$ which are adjacent.
For $u \in S'^1$, there is a path of length $4$ from $v$ to $u$ using exactly one vertex of each of $X^1$, $X^2$, $S'^2$. 
Thus, the diameter of~$G$ is~$4$.

We claim that $\mathcal{S}$ contains an exact $3$-cover of $X$ if and only if $G$ has a matching cut of size~$6q$.
First suppose that $\mathcal{S}$ contains an exact $3$-cover $\mathcal{C}$ of $X$. We colour every vertex of $K_X$ red. We colour a graph $K_S$ blue if $S \in \mathcal{C}$ and red otherwise.
This yields a valid red-blue colouring of value $6q$, and thus a matching cut of size $6q$.

Now suppose that $G$ has a matching cut $M$ of size $6q$. 
As $K_X$ and $K_S$ for all $S \in \mathcal{S}$ are complete bipartite graphs with at least $3$ vertices on each side, they are monochromatic, see Observation~\ref{o-c-all-kr-mono}.
Say $K_X$ is coloured red. Thus, exactly $q$ subgraphs $K_S$ are coloured blue. Moreover, no two blue $K_S$ have a common red neighbour in $K_X$. Hence, the blue~$K_S$ correspond to an exact $3$-cover of $X$. 
\end{proof}

\noindent
We can also give the analogue result for \mdpm{}.
Note that no changes to the construction are needed to make the maximum matching cut extendable to a perfect matching. Since we have an exact $3$-cover,  if there is a matching cut in the graph, then all vertices in $K_X$ are matched. Further, each $K_S$ is either completely matched to $K_X$ or no vertex of $K_S$ has a matching partner in this matching cut. In this case, we can match every vertex in $K_S$ with another vertex from the same $K_S$, since $K_S$ is a complete bipartite graph with partition classes of equal size. 
Thus, the proof of Theorem~\ref{t-np-mmc-biprad} leads to the following theorem.
\begin{theorem}\label{t-np-mdpm-biprad}
\mdpm\ is \NP-hard for graphs of radius at most~$3$ and diameter at most~$4$. 
\end{theorem}

\section{Conclusion}\label{s-con}
In this paper we investigated the complexity of matching cut variants on bipartite graphs of bounded radius and diameter. Our results lead to complexity dichotomies for \dc, \mmc{} and \mdpm. 
For \mc{} and \dpm{}, the dichotomy for bipartite graphs of bounded diameter was already complete. However, for bipartite graphs of bounded radius, few cases were missing to complete the dichotomy. We settled the case of \dpm{} on bipartite graphs of radius~$2$. Therefore, only the following two cases remain open.

\begin{open}
What is the computational complexity of \mc{} and \dpm{} for bipartite graphs of radius~$3$?
\end{open}

\noindent
For \pmc{} we gave the first result showing it is \NP-complete
for bipartite graphs of bounded radius and diameter. Only one
case remains open for the radius. The gap for the diameter is however still large.

\begin{open}
What is the computational complexity of \pmc{} for bipartite graphs of
	radius~$3$?
\end{open}

\begin{open}
What is the computational complexity of \pmc{} for bipartite graphs of diameter~$d$ with $4 \leq d \leq 9$?
\end{open}

\noindent
Especially settling the case of bipartite graphs of radius~$3$ would be interesting, since it is open for \mc, \dpm{} and \pmc. Therefore it is an interesting candidate to determine whether these problems differ in complexity on bipartite graphs of bounded radius and diameter.
\bibliography{ref}

\begin{thebibliography}{10}

\bibitem{BCD24}
Edouard Bonnet, Dibyayan Chakraborty, and Julien Duron.
\newblock Cutting barnette graphs perfectly is hard.
\newblock {\em Theoretical Computer Science}, page 114701, 2024.

\bibitem{Bo09}
Paul~S. Bonsma.
\newblock The complexity of the {M}atching-{C}ut problem for planar graphs and
  other graph classes.
\newblock {\em Journal of Graph Theory}, 62:109--126, 2009.

\bibitem{BP}
Valentin Bouquet and Christophe Picouleau.
\newblock The complexity of the {P}erfect {M}atching-{C}ut problem.
\newblock {\em Journal of Graph Theory}, to appear.

\bibitem{CHLLP21}
Chi{-}Yeh Chen, Sun{-}Yuan Hsieh, Ho{\`{a}}ng{-}Oanh Le, Van~Bang Le, and
  Sheng{-}Lung Peng.
\newblock Matching {C}ut in graphs with large minimum degree.
\newblock {\em Algorithmica}, 83:1238--1255, 2021.

\bibitem{Ch84}
Vasek Chv{\'{a}}tal.
\newblock Recognizing decomposable graphs.
\newblock {\em Journal of Graph Theory}, 8:51--53, 1984.

\bibitem{FLPR23}
Carl Feghali, Felicia Lucke, Dani\"el Paulusma, and Bernard Ries.
\newblock Matching cuts in graphs of high girth and {$H$}-free graphs.
\newblock {\em Proc.~ISAAC 2023, LIPIcs}, 283:28:1--28:16, 2023.

\bibitem{GS21}
Guilherme Gomes and Ignasi Sau.
\newblock Finding cuts of bounded degree: complexity, {F}{P}{T} and exact
  algorithms, and kernelization.
\newblock {\em Algorithmica}, 83:1677--1706, 2021.

\bibitem{Gr70}
Ronald~L. Graham.
\newblock On primitive graphs and optimal vertex assignments.
\newblock {\em Annals of the New York Academy of Sciences}, 175:170--186, 1970.

\bibitem{HT98}
Pinar Heggernes and Jan~Arne Telle.
\newblock Partitioning graphs into generalized dominating sets.
\newblock {\em Nordic Journal of Computing}, 5:128--142, 1998.

\bibitem{Ka72}
Richard~M. Karp.
\newblock Reducibility among {C}ombinatorial {P}roblems.
\newblock {\em Complexity of Computer Computations}, pages 85--103, 1972.

\bibitem{LL19}
Ho{\`{a}}ng-Oanh Le and Van~Bang Le.
\newblock A complexity dichotomy for {M}atching {C}ut in (bipartite) graphs of
  fixed diameter.
\newblock {\em Theoretical Computer Science}, 770:69--78, 2019.

\bibitem{LL23}
Ho{\`{a}}ng{-}Oanh Le and Van~Bang Le.
\newblock Complexity results for matching cut problems in graphs without long
  induced paths.
\newblock {\em Proc.~{WG} 2023, LNCS}, 14093:417--431, 2023.

\bibitem{LLPR24}
Van~Bang Le, Felicia Lucke, Daniël Paulusma, and Bernard Ries.
\newblock Maximizing matching cuts.
\newblock {\em Encyclopedia of Optimization}, to appear, 2025.

\bibitem{LT22}
Van~Bang Le and Jan~Arne Telle.
\newblock The {P}erfect {M}atching {C}ut problem revisited.
\newblock {\em Theoretical Computer Science}, 931:117--130, 2022.

\bibitem{LMPS24}
Felicia Lucke, Ali Momeni, Dani{\"{e}}l Paulusma, and Siani Smith.
\newblock Finding $d$-{C}uts in graphs of bounded diameter, graphs of bounded
  radius and {$H$}-free graphs.
\newblock {\em submitted manuscript}, 2024.

\bibitem{LPR22}
Felicia Lucke, Dani\"el Paulusma, and Bernard Ries.
\newblock On the complexity of {M}atching {C}ut for graphs of bounded radius
  and ${H}$-free graphs.
\newblock {\em Theoretical Computer Science}, 936, 2022.

\bibitem{LPR23a}
Felicia Lucke, Dani{\"{e}}l Paulusma, and Bernard Ries.
\newblock Finding matching cuts in ${H}$-free graphs.
\newblock {\em Algorithmica}, 85:3290--3322, 2023.

\bibitem{LPR24}
Felicia Lucke, Dani\"{e}l Paulusma, and Bernard Ries.
\newblock {Dichotomies for Maximum Matching Cut: {$H$}-Freeness, Bounded
  Diameter, Bounded Radius}.
\newblock {\em Theoretical Computer Science}, 1017, 2024.

\bibitem{Sc78}
Thomas~J. Schaefer.
\newblock The complexity of satisfiability problems.
\newblock {\em Proc.~STOC 1978}, pages 216--226, 1978.

\end{thebibliography}

\end{document}